\documentclass[a4paper, 10pt]{amsart}
\usepackage[a4paper, top=3.5cm, bottom=3.5cm]{geometry}
\usepackage[utf8]{inputenc}
\usepackage[T1]{fontenc}
\usepackage{ae,aecompl}
\usepackage{fancyhdr}
\usepackage{amsmath}
\usepackage{amssymb}
\usepackage{amsthm}
\usepackage{array}
\usepackage[all]{xy}
\usepackage{bm}
\usepackage{graphicx}
\usepackage{enumerate}
\usepackage{listings}
\usepackage{caption}
\usepackage{subcaption}
\usepackage{mathrsfs}
\usepackage{indentfirst}

\allowdisplaybreaks

\newtheorem{twr}{Theorem}[section]
\newtheorem{theorem}[twr]{Theorem}
\newtheorem{lem}[twr]{Lemma}

\newtheorem{stw}[twr]{Proposition}
\newtheorem{pbm}[twr]{Problem}
\newtheorem{col}[twr]{Corollary}
\newtheorem{con}[twr]{Conjecture}
\newtheorem{que}[twr]{Question}

\theoremstyle{definition}
\newtheorem{exple}[twr]{Example}

\theoremstyle{remark}
\newtheorem*{rem}{Remark}

\makeatletter
\renewcommand\@seccntformat[1]{\csname the#1\endcsname.\quad}

\makeatother

\DeclareMathOperator{\ovb}{v}%{\overline{{\it b}}}
\DeclareMathOperator{\ovbb}{\textbf{v}}%{\overline{\textbf{b}}}
\DeclareMathOperator{\bb}{\textnormal{\textbf{b}}}
\DeclareMathOperator{\cc}{\textnormal{\textbf{c}}}
\DeclareMathOperator{\dd}{\textnormal{\textbf{d}}}
\DeclareMathOperator{\degu}{\overline{\deg}}
\DeclareMathOperator{\degl}{\underline{\deg}}
\DeclareMathOperator{\lvl}{lvl}
\DeclareMathOperator{\rk}{r}

\address{Jagiellonian University \\
Faculty of Mathematics and Computer Science \\
Institute of Mathematics \\
\L ojasiewicza\ 6 \\
30-348 Krak\'ow \\
Poland}
\email{blazejz@poczta.onet.pl}


\begin{document}

\title[Recurrence sequences]{Recurrence sequences connected with the $m$--ary partition function and their divisibility properties}
\author{Błażej Żmija}

\begin{abstract}
In this paper we introduce a class of sequences connected with the $m$--ary partition function and investigate their congruence properties. In particular, we get facts about the sequences of $m$--ary partitions $(b_{m}(n))_{m\in\mathbb{N}}$ and $m$--ary partitions with no gaps $(c_{m}(n))_{m\in\mathbb{N}}$. We prove, for example, that for any natural number $2<h\leq m+1$ in both sequences $(b_{m}(n))_{m\in\mathbb{N}}$ and $(c_{m}(n))_{m\in\mathbb{N}}$ any residue class modulo $h$ appears infinitely many times.

Moreover, we give new proofs of characterisations modulo $m$ in terms of base--$m$ representation of $n$ for sequences $(b_{m}(n))_{m\in\mathbb{N}}$ and $(c_{m}(n))_{m\in\mathbb{N}}$. We also present a general method of finding such characterisations modulo any power of $m$. Using our approach we get description of $(b_{m}(n)\mod{\mu_{2}})_{n\in\mathbb{N}}$, where $\mu_{2}=m^{2}$ if $m$ is odd and $\mu_{2}=m^{2}/2$ if $m$ is even.
\end{abstract}

\maketitle

\section{Introduction}

For any fixed natural number $m\geq 2$ and any $n\in\mathbb{N}_{0}$ (we use the notation $\mathbb{N}=\{1,2,3,\ldots \}$ and $\mathbb{N}_{0}:=\mathbb{N}\cup\{0\}$) let $b_{m}(n)$ denote the number of the $m$--ary partitions of $n$, i.e. the number of representations of the integer $n$ in the form $n=n_{0}+n_{1}m+n_{2}m^{2}+\ldots +n_{t}m^{t}$, where $n_{0},\ldots ,n_{t}\in\mathbb{N}_{0}$. We denote $\textbf{b}_{m}=(b_{m}(n))_{n\in\mathbb{N}_{0}}$. We call the number $m$ {\it the degree of the sequence} $\textbf{b}_{m}$. We explain this name later.

There is a well known recurrence relation for $b_{m}(n)$:
\begin{equation}\label{rekur}
\left\{ 
\begin{array}{lr}
b_{m}(mn)=b_{m}(mn+1)=\ldots =b_{m}(mn+(m-1)), & n\geq 0, \\
b_{m}(mn)-b_{m}(mn-1)=b_{m}(n), & n\geq 0, \\
b_{m}(0)=1.
\end{array} 
\right.
\end{equation}
Let us denote the first equality in (\ref{rekur}) by (\ref{rekur}a) and the second one by (\ref{rekur}b). 

Congruence properties of $b_{m}(n)$ were firstly studied (in the case $m=2$) by Churchhouse in \cite{Chu}. Among other things, he characterized the numbers $n$ for which $b_{2}(n)\equiv 0\pmod 4$ and proved that none of the numbers $b_{2}(n)$ is divisible by $8$. This result was strengthened for example by Alkauskas \cite{Alk} who gave a formula for $b_{2}(n)\mod 32$ for any $n\in\mathbb{N}$.

In the cited paper Churchhouse conjectured that for any odd number $n$ the congruence $b_{2}(2^{s+2}n)-b_{2}(2^{s}n)\equiv 2^{\mu(s)}\pmod {2^{\mu(s)+1}}$ holds with $\mu(s)=\lfloor\frac{3s+4}{2}\rfloor$. This was independently proved by R{\o}dseth \cite{Rod} and Gupta \cite{Gup2}. 

The theorems mentioned above may be extended to the case of any degree $m$. This was done for example by Andrews, Fraenkel and Sellers in \cite{And3}. They showed that if $n=a_{0}+a_{1}m+\ldots +a_{t}m^{t}$ where $0\leq a_{j}\leq m-1$ for $j=0,\ldots ,t$, then $b_{m}(mn)\equiv\prod_{j=1}^{t}(a_{j}+1)\pmod m$. It follows from (\ref{rekur}a) that this result characterizes the numbers $b_{m}(n)\mod m$ for any $n\in\mathbb{N}$ (not necessarily divisible by~$m$). However, it is important to notice that the same result was obtained for 15 years earlier by Alkauskas in the Lithuanian version of his paper \cite{Alk2}, and his proof is easier and shorter than the one given by Andrews {\it et. al}. A different proof was given by Edgar in \cite{Edg}.

Another interesting result was obtained by R{\o}dseth whose theorem, after generalisations done by Andrews \cite{And1} and Gupta \cite{Gup1}, takes the form $b_{m}(m^{r+1}n)\equiv b_{m}(m^{r}n)\pmod{\frac{m^{r}}{\gamma_{r}}}$ for all $m\geq 2$ and $r,\ n\geq 1$. Here $\gamma_{r}=1$ if $m$ is odd and $\gamma_{r}=2^{r-1}$ if $m$ is even.

We also consider a special class of the $m$--ary partitions, namely those for which any positive integer $j$ satisfies the following condition: if $n_{j}>0$, then $n_{j-1}>0$. We call them the $m$--ary partitions with no gaps and denote their number by $c_{m}(n)$ and write $\textbf{c}_{m}=(c_{m}(n))_{n\in\mathbb{N}_{0}}$. Similarly as in case of the sequence $\textbf{b}_{m}$, we call the number $m$ the degree of the sequence $\textbf{c}_{m}$.

Andrews {\it et. al.} found in \cite{And4} a recurrence relation for the sequence $\cc_{m}$. We rewrite these equations in an equivalent but more accurate form:
\begin{equation}\label{rekur2}
\left\{ 
\begin{array}{lr}
c_{m}(mn+1)=c_{m}(mn+2)=\ldots =c_{m}(mn+m), & n\geq 0, \\
c_{m}(mn+1)-c_{m}(mn)=c_{m}(n), & n\geq 1, \\
c_{m}(0)=c_{m}(1)=1. &
\end{array} 
\right.
\end{equation}
Similarly as before, we denote the first equation in (\ref{rekur2}) by (\ref{rekur2}a) and the second by (\ref{rekur2}b).

The congruence $c_{m}(n)\equiv 0\pmod m$ was considered by Andrews {\it et. al.} in \cite{And4}. They found a characterization of solutions of this congruence similar to the one obtained in the case of $\textbf{b}_{m}$. In \cite{And2} Andrews {\it et. al.} found infinite families of solutions of the congruence $c_{m}(n)\equiv 0\pmod {m^{j}}$ for $m$ prime and $0\leq j<m$. This result was generalized by Hou {\it et. al.} in \cite{Hou} to the congruence $c_{m}(n)\equiv 0\pmod {\frac{m^{j}}{\gamma_{j}}}$, where $m$ is an arbitrary natural number, $0\leq j<m$ and $\gamma_{j}=1$ if $m$ is odd and $\gamma_{j}=2^{j-1}$ if $m$ is even.

There is a huge number of theorems of this type. Some of them can be found for example in \cite{Dir} and \cite{RodS2}. The common thing about these results %(except those containing hyper $m$--ary partitions) 
is that they investigate the solutions of the congruence modulo $h$ only for $h$ somehow connected with $m$ (i.e. $h$ is either a divisor of $m$, a power of $m$ or $m$ itself). In this paper we study a more general case without any links between $h$ and $m$ (except for some bounds on $h$ which depend on $m$). Alkauskas was the first (and probably the last up to this paper) to do this. In the paper mentioned before he proved (by performing some numerical computations) that the congruence $b_{2}(n)\equiv 0\pmod h$ has infinitely many solutions in $n\in\mathbb{N}$ for any $2\leq h\leq 14$, $h\neq 8$. He made his computations about 15 years ago and we have improved them. In the next section we also define a new class of recurrence sequences $(b_{T}(n))_{n\in\mathbb{N}_{0}}$ which generalizes all sequences mentioned before. For these objects we investigate the congruence $b_{T}(n)\equiv t\pmod h$ for some natural numbers $h$ and $0\leq t\leq h-1$.

We also give characterisations for special cases of sequences $(b_{T}(n))_{n\in\mathbb{N}_{0}}$ modulo some numbers. In particular, we reprove theorems of Alkauskas and Andrews {\it et.al.} characterising numbers $b_{m}(n)$ and $c_{m}(n)$ modulo $m$ and present a general method which allows us to get similar characterisations modulo any power of $m$ and use it to describe the sequence $(b_{m}(n)\mod{\mu_{2}})_{n\in\mathbb{N}_{0}}$, where $\mu_{2}=m^{2}$ if $m$ is odd and $\mu_{2}=m^{2}/2$ if $m$ is even. 

We consider the sequence $\ovbb_{m}=(\ovb_{m}(n))_{n\in\mathbb{N}_{0}}$ given by:
\begin{equation}\label{rekurover}
\left\{ 
\begin{array}{lr}
\ovb_{m}(mn)=\ovb_{m}(mn+1)=\ldots =\ovb_{m}(mn+(m-1)), & n\geq 0, \\
\ovb_{m}(mn)-\ovb_{m}(mn-1)=\ovb_{m}(n)+\ovb_{m}(n-1), & n\geq 1, \\
\ovb_{m}(0)=1.
\end{array} 
\right.
\end{equation}
We again call the number $m$ the degree of the sequence $\ovbb_{m}$. 
We introduce the aforementioned sequence due to the elegant and transparent way its properties work with some of the theorems obtained here. For example we find a characterisation modulo $m$ for it and obtain some partial results modulo $m^{2}$.

\section{Notation and some results}

Let $K=(k_{n})_{n\in\mathbb{N}_{0}}$ be a sequence of natural numbers such that $0=k_{0}<k_{1}<k_{2}<k_{3}<\ldots$ and let $L=(l_{i})_{i=0}^{t}$ and $R=(r_{i})_{i=0}^{u}$ be finite sequences of integers. For a triple $T=(K,L,R)$ we define a new sequence $\textbf{b}_{T}=(b_{T}(n))_{n\in\mathbb{N}_{0}}$ as follows:
\begin{equation}\label{rekurT}
\left\{
\begin{array}{l}
b_{T}(k_{n})=b_{T}(k_{n}+1)=\ldots =b_{T}(k_{n+1}-1), \\
\\
\sum_{i=0}^{t}l_{i}b_{T}(k_{n-i})=\sum_{j=0}^{u}r_{j}b_{T}(n-j).
\end{array}
\right.
\end{equation}
We denote the first equality in (\ref{rekurT}) by (\ref{rekurT}a) and the second by (\ref{rekurT}b). Notice that we do not specify the first values of the sequence $(b_{T}(n))_{n\in\mathbb{N}_{0}}$. However, in the rest of the paper we do not need this information. Thus, any triple $T$, in fact, generates infinitely many sequences $\textbf{b}_{T}$, but theorems proved in the sequel are true for all of them.

Let us introduce the numbers: 
\begin{displaymath}
l=\sum_{i=0}^{t}l_{i} \text{\ \ \ and\ \ \ } r=\sum_{j=0}^{u}r_{i}.
\end{displaymath}

We also define the {\it upper} and {\it lower degree} of the sequence $K$:
\begin{displaymath}
\begin{split}
\degu K:= & \sup\{\ d\in\mathbb{N}\ |\ \text{there exist infinitely many $n\in\mathbb{N}$ such that } k_{n}-k_{n-1}\geq d\},\\
\degl K:= & \inf\{\ d\in\mathbb{N}\ |\ \text{there exist infinitely many $n\in\mathbb{N}$ such that } k_{n}-k_{n-1}=d\}.
\end{split}
\end{displaymath}
If $\degu K =\degl K$, we denote this common value by $\deg K$ and call it the {\it degree} of the sequence $K$. We also use the notation of the degree (or upper/lower degree) of the sequence $\textbf{b}_{T}$, which is the degree (or upper/lower degree respectively) of the sequence $K$ from the corresponding triple $T$.

\begin{exple}\label{exple1}
The sequences mentioned in the introduction are a special case of $b_{T}$ for the following choices of $T$ (up to the first values):

1. if $K=(mn)_{n=0}^{\infty}$, $L=(1,-1)$ and $R=(1)$, then we get $\textbf{b}_{m}$,

2. if $K$ is defined by $k_{0}=0$ and $k_{n}=mn+1$ for $n\geq 1$, $L=(1,-1)$ and $R=(1)$, then we get~$\textbf{c}_{m}$,

3. if $K=(mn)_{n=0}^{\infty}$, $L=(1,-1)$ and $R=(1,1)$, then we get $\ovbb_{m}$.

For all these sequences we have $\deg K=m$, which explains why we called the number $m$ the degree.
\end{exple}

Now we are ready to give the following result.

\begin{theorem}\label{TwMain}
Let $T=(K,L,R)$ be as above. Let $h\in\mathbb{N}$ be co--prime to $r=\sum_{j=0}^{u}r_{j}$ and satisfy $2\leq h\leq \degl K-u+1$. Then at least one of the following holds:
\begin{enumerate}
\item[\textsc{1)}] All but finitely many values of the sequence $\bb_{T}$ are divisible by $h$;
\item[\textsc{2)}] All but finitely many values of the sequence $\bb_{T}$ are not co--prime to $h$ and the first condition is not satisfied;
\item[\textsc{3)}] For any sequence $\dd_{T}=(d_{T}(n))_{n\in\mathbb{N}_{0}}$ such that 
\begin{displaymath}
d_{T}(k_{n}+s)-d_{T}(k_{n}+s-1)=l_{0}b_{T}(k_{k_{n}+s})+l_{1}b_{T}(k_{k_{n}+s-1})+\ldots +l_{t}b_{T}(k_{k_{n}+s-t})
\end{displaymath} 
for $n\in\mathbb{N}_{0}$ and $s=0,\ldots ,k_{n+1}-k_{n}-1$ any residue class modulo $h$ appears in the sequence $(d_{T}(n)\mod h)_{n\in\mathbb{N}_{0}}$ infinitely many times.
\end{enumerate}
Moreover, only conditions \textsc{2)} and \textsc{3)} can be satisfied together.
\end{theorem}

\begin{proof}
Assume that the first case does not hold. That is, there exist infinitely many natural numbers $n$ for which $b_{T}(n)$ is not divisible by $h$. It is enough to show that if $b_{T}(k_{n})$ is co--prime to $h$ for some sufficiently large $n$, then any residue class modulo $h$ appears among numbers $d_{T}(k_{n}+s)\mod h$ for $s=u-1,\ldots ,u+h-1$. Assume $b_{T}(k_{n})\equiv c\pmod h$ for certain $c\in\mathbb{Z}/h\mathbb{Z}$. Thus, by (\ref{rekurT}a), we have:
\begin{displaymath}
\begin{split}
b_{T}(k_{n})\equiv &\ c\pmod h, \\
b_{T}(k_{n}+1)\equiv &\  c\pmod h, \\
\vdots\ & \\
b_{T}(k_{n+1}-1)\equiv &\  c\pmod h. \\
\end{split}
\end{displaymath}
Observe that for $s=u,\ldots ,u+h-2$ one has:
\begin{displaymath}
\begin{split}
d_{T}(k_{n}+s)-d_{T}(k_{n}+s- & 1)=l_{0}b_{T}(k_{k_{n}+s})+l_{1}b_{T}(k_{k_{n}+s-1})+\ldots +l_{t}b_{T}(k_{k_{n}+s-t}) \\
= & r_{0}b_{T}(k_{n}+s)+r_{1}b_{T}(k_{n}+s-1)+\ldots +r_{u}b_{T}(k_{n}+s-u) \\
\equiv & rc \pmod h,
\end{split}
\end{displaymath}
where we used the assumption $u+h-2\leq\degl K -1$ in the equivalence on the end. Hence, we simply get the congruence:
\begin{displaymath}
d_{T}(k_{n}+s)\equiv d(k_{n}+u-1)+(s-u+1)rc \pmod h.
\end{displaymath}
We know that $rc$ is co--prime to $h$ and this, together with the last congruence, implies the statement of the theorem.
\end{proof}

\begin{rem}
In the previous theorem we can take the sequence $\dd_{T}$ given by:
\begin{equation}\label{seqd}
d_{T}(k_{n}+s):=\sum_{i=0}^{t}\Big(\sum_{j=0}^{i}l_{j}\Big)b_{T}(k_{k_{n}+s-i})+l[b_{T}(k_{k_{n}-t})+\ldots +b_{T}(k_{k_{n}+s-1-t})]
\end{equation}
for $n\in\mathbb{N}_{0}$ and $s=0,\ldots ,k_{n+1}-k_{n}-1$, with the convention that for $s=0$ the expression in the square brackets disappears. Moreover, if $\dd'_{T}=(d'_{T}(n))_{n\in\mathbb{N}_{0}}$ is another sequence satisfying condition from Theorem \ref{TwMain}, then for any $j\in\mathbb{N}_{0}$ we have:
\begin{displaymath}
d'_{T}(j+1)-d'_{T}(j)=d_{T}(j+1)-d_{T}(j).
\end{displaymath}
Summing these equalities on both sides for $j=0,\ldots ,n-1$, we get:
\begin{displaymath}
d'_{T}(n)=d_{T}(n)+(d'_{T}(0)-d_{T}(0)).
\end{displaymath}
Hence, any sequence $\dd_{T}$ from the condition 3) of Theorem \ref{TwMain} may differ from the sequence (\ref{seqd}) only by a constant term.
\end{rem}

Firstly, we would like to have some criteria allowing us to say which conditions of Theorem \ref{TwMain} hold. We prove the following useful:

\begin{lem}\label{LemCrit}
Let $2\leq p\leq \degl K-u+1$ be a prime number not dividing $r$ and $r_{u}$. Assume the condition \textsc{1)} from the statement of {\rm Theorem \ref{TwMain}} is satisfied and also %$b_{T}(k_{n+1})>b_{T}(k_{n})$ and  
$k_{n+u-t}\geq n$ for any $n\in\mathbb{N}_{0}$. Then one of the next conditions holds:
\begin{enumerate}
\item[\textsc{1)}] p divides $b_{T}(n)$ for any $n\in\mathbb{N}_{0}$,
\item[\textsc{2)}] If $n_{0}$ is the largest such that $p\nmid b_{T}(n_{0})$ then $n_{0}=k_{n_{0}+u-t}$ and $l_{t}\equiv r_{u}\pmod p$.
\end{enumerate}
\end{lem}

\begin{proof}
If there exists $n_{0}$ as above, then (\ref{rekurT}b) for $n=n_{0}+u$ simplifies to $l_{t}b_{T}(k_{n_{0}+u-t})\equiv r_{u}b_{T}(n_{0})\pmod p$. If $k_{n+u-t}>n$, then $r_{u}b_{T}(n_{0})\equiv 0\pmod p$, which leads to contradiction. Thus, $k_{n_{0}+u-t}\leq n_{0}$, but by our assumption $k_{n_{0}+u-t}\geq n_{0}$ so $k_{n_{0}+u-t}=n_{0}$. Hence, $b_{T}(k_{n_{0}+u-t})=b_{T}(n_{0})$ and $l_{t}\equiv r_{u}\pmod p$.
\end{proof}

\begin{rem}
For $h$ not being prime in the second case of the statement of previous lemma we get only the information that $l_{t}b_{T}(k_{n_{0}+u-t})\equiv r_{u}b_{T}(n_{0})\pmod h$.
\end{rem}

Next, we would also like to have ``exactly one" instead of ``at least one" in Theorem \ref{TwMain}. We do not know if we can have it in general but, fortunately, this condition can by easily seen to be satisfied in the important case of $L=(1,-1)$ (as we mentioned in Example \ref{exple1} this holds for all sequences presented in the Introduction) with the sequence $\dd_{T}$ given by formula (\ref{seqd}) which simplifies to $d_{T}(n)=b_{T}(n)$ for any $n\in\mathbb{N}_{0}$. Now we concentrate on this case.
Another natural restriction is to consider only the congruence $b_{T}(n)\equiv 0\pmod h$ instead of $b_{T}(n)\equiv s\pmod h$ for any $0\leq s\leq h-1$. We also focus for a while on $h$ being a prime number. Equipped with these assumptions we can get better results.

\begin{theorem}\label{TwL}
Let $L=(1,-1)$ and $p$ be a prime number satisfying $3\leq p\leq \degu K -u+1$ and $p\nmid r$. Then the congruence $b_{T}(n)\equiv 0\pmod p$ has infinitely many solutions in $\mathbb{N}$.
\end{theorem}

\begin{proof}
Assume that $T$ and $p$ are as in the statement of the theorem but there exist only finitely many solutions of the congruence. Let $n_{0}$ be the largest such solution (if there are no solutions, take $n_{0}=0$). Let $n>n_{0}$ be such that $d:=k_{n+1}-k_{n}\geq p+u-1$. More precisely, we take $n$ such that $d=\degu K$ if $\degu K$ is finite and such that $d\geq p+u-1$ is any natural number if $\degu K$ is infinite. Let $b_{T}(k_{n})\equiv c\pmod p$ for some $c\in\mathbb{Z}/p\mathbb{Z}$. Then, by (\ref{rekurT}a), we get:
\begin{equation}\label{eq2}
\begin{split}
b_{T}(k_{n}) & \equiv c\pmod p, \\
b_{T}(k_{n}+1) & \equiv c\pmod p, \\
 & \ \ \vdots \\
b_{T}(k_{n}+d-1) & \equiv c\pmod p.
\end{split}
\end{equation}
Now consider numbers of the form $b_{T}(k_{k_{n}+t-1})$ for $t=u,\ldots ,d-1$ (we know that $d-1>u$, because $d\geq p+u-1\geq u+2$). Obviously, none of them is divisible by $p$. If one of them is equal to $-rc$ modulo $p$ (which is non--zero by our assumptions), then by (\ref{rekurT}b): 
\begin{displaymath}
\begin{split}
b_{T}(k_{k_{n}+t})= & b_{T}(k_{k_{n}+t-1})+[r_{0}b_{T}(k_{n}+t)+\ldots +r_{u}b_{T}(k_{n}+t-u)] \\ 
 \equiv & -rc+rc\equiv 0\pmod p,
\end{split}
\end{displaymath}
which is a contradiction. Thus, we have at least $p-1\leq d-u$ numbers taking at most $p-2$ different values modulo $p$. Hence, there exist numbers $u\leq t_{0}<t_{1}\leq p+u-2$ for which $b_{T}(k_{k_{n}+t_{0}-1})$ and $b_{T}(k_{k_{n}+t_{1}-1})$ are equal modulo $p$. Thus:
\begin{displaymath}
\begin{split}
0 \equiv & b_{T}(k_{k_{n}+t_{1}-1})-b_{T}(k_{k_{n}+t_{0}-1}) \\
= & [b_{T}(k_{k_{n}+t_{1}-1})-b_{T}(k_{k_{n}+t_{1}-2})]+[b_{T}(k_{k_{n}+t_{1}-2})-b_{T}(k_{k_{n}+t_{1}-3})]+ \\
  & \ldots +[b_{T}(k_{k_{n}+t_{0}+1})-b_{T}(k_{k_{n}+t_{0}})]+[b_{T}(k_{k_{n}+t_{0}})-b_{T}(k_{k_{n}+t_{0}-1})] \\
= & [r_{0}b_{T}(k_{n}+t_{1}-1)+\ldots +r_{u}b_{T}(k_{n}+t_{1}-1-u)]+[r_{0}b_{T}(k_{n}+t_{1}-2)+ \\ 
  & \ldots +r_{u}b_{T}(k_{n}+t_{1}-2-u)]+\ldots +[r_{0}b_{T}(k_{n}+t_{0})+\ldots +r_{u}b_{T}(k_{n}+t_{0}-u)]) \\
\equiv & (t_{1}-t_{0})rc \pmod p.
\end{split}
\end{displaymath}
We have $0<t_{1}-t_{0}<p$ and $rc$ is invertible in $\mathbb{Z}/p\mathbb{Z}$, so $(t_{1}-t_{0})rc$ cannot be equal to $0$ modulo $p$. This gives a contradiction and finishes the proof.
\end{proof}

\begin{rem}
If $\degl K=\degu K$, then Theorem \ref{TwL} follows from Theorem \ref{TwMain}. However, it can be easily seen that these two expressions may differ in general. For example, taking the sequence $K$ defined recursively by:
\begin{displaymath}
k_{n}=
\left\{
\begin{array}{ll}
 k_{n-1}+1 & \textrm{if n is odd,} \\
 k_{n-1}+n & \textrm{if n is even,}
\end{array}
\right.
\end{displaymath} 
with $k_{0}=0$ we get $\degl K=1$ and $\degu K=\infty$.
\end{rem}

Observe that for the sequences of the $m$--ary partitions, the $m$--ary partitions with no gaps and the sequence $\ovbb_{m}$ the difference $k_{n+1}-k_{n}$ is constant for all $n>0$. The next theorem says something about sequences satisfying a similar condition.

\begin{theorem}\label{TwLd}
Let $L=(1,-1)$. Assume that for any $n$ greater than some number $N$ the difference ${k_{n+1}-k_{n}=:d}$ is a constant greater than $u$ and let $3\leq p\leq d^{2}+d-u+1$ be a prime number. If there exists $n_{1}\geq k_{N+1}$ such that $b_{T}(n_{1})\equiv 0\pmod p$, then the congruence $b_{T}(n)\equiv 0\pmod p$ has infinitely many solutions for $n\in\mathbb{N}$.
\end{theorem}

\begin{proof}
Observe that the range of $p$ for which the method used in the previous proof works is closely connected with the number of consecutive values of $b_{K}$, which are equal modulo $p$. In the above proof we were able to produce only $\degu K$ such numbers. Now we want to use the existence of solution of the congruence $b_{K}(n)\equiv 0\pmod p$ to admit larger values of $p$ and prove our result.

Similarly as before, we assume that there exist only finitely many solutions of the congruence. Let $n_{0}$ be the largest one. Then, by (\ref{rekurT}a), it is of the form $n_{0}=n'+(d-1)$ for some $n'\in\mathbb{N}$. Hence, we have:
\begin{displaymath}
\begin{split}
b_{K}(n') & \equiv 0\pmod p, \\
b_{K}(n'+1) & \equiv 0\pmod p, \\
 & \ \ \vdots \\
b_{K}(n'+(d-1)) & \equiv 0\pmod p. \\
\end{split}
\end{displaymath}
Thus, by (\ref{rekurT}a) and (\ref{rekurT}b), the following equalities hold:
\begin{displaymath}
\begin{split}
b_{T}(k_{n'-1})\equiv & \ldots\equiv b_{T}(k_{n'}-1)\equiv b_{T}(k_{n'})\equiv\ldots \\
\equiv & b_{T}(k_{n'+1}-1)\equiv b_{T}(k_{n'+1})\equiv\ldots \\
\equiv & b_{T}(k_{n'+(d-1)}+(d-1))\pmod p.
\end{split}
\end{displaymath}
Hence, we get $d^{2}+d$ equal values modulo $p$ and repeating the previous reasoning we obtain the statement of our theorem.
\end{proof}

We consider the case $k=3$ separately using a different method. This allows us to gain a more general result in the following form.

\begin{theorem}\label{Twp3}
Let $L=(1,-1)$ and assume that $3\nmid r$. If $k_{n+1}-k_{n}\geq u+1$ then at least one of the numbers $b_{T}(k_{n})$, $b_{T}(k_{k_{n}+u-1})$, $b_{T}(k_{k_{n}+u})$ and $b_{T}(k_{k_{n}+u+1})$ is divisible by $3$.
\end{theorem}

\begin{proof}
Choose a number $n\in\mathbb{N}$ such that $k_{n+1}-k_{n}\geq u+2$. Then $b_{T}(k_{n})=b_{T}(k_{n}+1)=\ldots =b_{T}(k_{n}+u+1)$. Assume that $b_{K}(k_{n})=c$ and $b_{K}(k_{k_{n}+u-1})=d$ for some $c$ and $d$. Then by (\ref{rekurT}) we have:
\sloppy
\begin{displaymath}
\begin{split}
 & b_{T}(k_{k_{n}+u})=b_{T}(k_{k_{n}+u-1})+\sum_{j=0}^{u}r_{j}b_{T}(k_{n}+u-j)\equiv d+rc \pmod 3, \\
 & b_{T}(k_{k_{n}+u+1})=b_{T}(k_{k_{n}+u})+\sum_{j=0}^{u}r_{j}b_{T}(k_{n}+u+1-j)\equiv d+2rc \pmod 3.
\end{split}
\end{displaymath} 
\fussy
It is easy to see that for any integers $c$, $d$ and $r$ with $3\nmid r$ at least one of the numbers $c$, $d$, $d+rc$ and $d+2rc$ is equal to $0$ modulo $3$. This completes the proof.
\end{proof}

Now we focus our attention on sequences with $L=(1,-1)$ and $R=(1)$. This assumption may look very restrictive but is satisfied by sequences of $m$--ary partitions and $m$--ary partitions with no gaps. 

We have performed computer computations which improved the theorem of Alkauskas. Our method (which is presented in detail in Section 4) is a bit similar to the one used by him. However, we were able to get results which are true not only for the sequence $\bb_{2}$ but for all sequences with $\deg K=2$. 

\begin{theorem}\label{comgen}
Let $L=(1,-1)$ and $R=(1)$. If $\deg K=2$, then for any $2\leq h\leq 41$, $4\nmid h$, there exist infinitely many $n$ such that $b_{T}(n)$ is divisible by $h$.
\end{theorem}

A modification of the method allowed us to omit the problem with values of $h$ divisible by $4$ in the case $\deg K=3$.

\begin{theorem}\label{comgen2}
Let $L=(1,-1)$ and $R=(1)$. If $\deg K=3$, then for any $2\leq h\leq 62$, $h\neq 36, 42, 45, 48, 54, 56, 60$, there exist infinitely many $n$ such that $b_{T}(n)$ is divisible by $h$.
\end{theorem}

In fact, in the case $\bb_{T}=\bb_{2}$ we have a bit more. More precisely, we have the following fact.

\begin{theorem}\label{twcom}
For any $3\leq h\leq 41$, $8\nmid h$, or $h\in\{44, 52, 60, 68, 76\}$ there exist infinitely many $n$ such that $b_{2}(n)$ is divisible by~$h$.
\end{theorem}

The condition $4\nmid h$ instead of $8\nmid h$ in Theorem \ref{comgen} follows from the fact that we have (by the Churchhouse Theorem) the information about the parity of the numbers $b_{2}(n)$. In general, the numbers $b_{T}(n)$ may not be even for infinitely many $n\in\mathbb{N}_{0}$. Now we give an example of this type.

\begin{exple}
Let us consider the sequence $K$ for which we get the following sequence $b_{T}$:
\begin{displaymath}
\left\{
\begin{array}{lr}
b_{T}(2n)-b_{T}(2n-1)=b_{T}(n), & n\geq 4, \\
b_{T}(2n+1)=b_{T}(2n), & n\geq 4, \\
b_{T}(0)=\ldots =b_{T}(7)=1. & 
\end{array}
\right.
\end{displaymath}
Then, obviously, $\degl K=2$. We want to show that for any $n\in\mathbb{N}$ the numbers of the form $b_{T}(4n+2)$ are odd. This is a simple consequence of the equality $b_{K}(6)=1$ and the following computations:
\begin{displaymath}
\begin{split}
b_{T}(4n+6)= & b_{T}(4n+4)+b_{T}(2n+3) \\
= & b_{T}(4n+2)+2b_{T}(2n+2) \\
\equiv & b_{T}(4n+2)\pmod 2.
\end{split}
\end{displaymath}
\end{exple}

In the next section we show some applications of the presented theorems. In Section 4 we present a more precise results on the sequences $\bb_{2}$ and $\bb_{3}$, which allow us to get information about the values for which we can say that at least one is divisible by $h$.

\section{Applications}

In this section we are mainly interested in the sequences $\bb_{T}$ with $L=(1,-1)$.  We start with the following, much stronger, form of Theorem \ref{TwMain}.

\begin{theorem}\label{ColMain}
Let $L=(1,-1)$ and $h\in\mathbb{N}$ be co--prime to $r=\sum_{j=0}^{u}r_{j}$ and satisfy $2\leq h\leq \degl K-u+1$. If $n<k_{n}$ and $u\leq k_{n}-k_{n-1}-1$ for any $n>0$ then exactly one of the following holds:
\begin{enumerate}
\item[\textsc{1)}] All values of the sequence $\bb_{T}$ are divisible by $h$,
\item[\textsc{2)}] All values of the sequence $\bb_{T}$ are not co--prime to $h$ and infinitely many values of $\bb_{T}$ are not divisible by $h$,
\item[\textsc{3)}] Any residue class modulo $h$ appears in the sequence $(b_{T}(n)\mod h)_{n\in\mathbb{N}_{0}}$ infinitely many times.
\end{enumerate}
\end{theorem}

\begin{proof}
Let us begin with the following observation: in the proof of Theorem \ref{TwMain} we showed that if $b_{T}(n)$ is co--prime to $h$ for some $n$, then any residue class modulo $h$ is represented by a number of the form $d_{T}(k_{n}+s)\mod h$ for $s=u-1,\ldots ,u+h-1$. But if $L=(1,-1)$, then the sequence $\dd_{T}$ given by (\ref{seqd}) becomes equal to the sequence $\bb_{T}$. Thus having one number $n$ for which $b_{T}(n)$ is co--prime to $h$ we can produce all residue classes modulo $h$ in the sequence $\bb_{T}$. In particular, we find a number of the form $k_{n}+s_{0}$ greater than $n$ for which $b_{T}(k_{n}+s_{0})$ is co--prime to $h$. We can apply the same reasoning to it and again produce all residue classes modulo $h$. Repeating this infinitely many times we get that any residue class modulo $h$ appears in the sequence $\bb_{T}$ infinitely many times.

Having the above information it is easy to deduce that if 1) and 2) does not hold, then 3) holds, and if 1) and 3) does not hold, then 2) holds. 

If 2) and 3) are not satisfied, then all but finitely many values of the sequence $\bb_{T}$ are divisible by $h$. Let $n_{0}$ be the largest number such that $b_{T}(n_{0})$ is not divisible by $h$. Then $n_{0}=k_{n'+1}-1$ for some $n'$. By the assumption we also have $u\leq k_{n_{0}+1}-k_{n_{0}}-1$. Hence, $b_{T}(n_{0})=b_{T}(n_{0}-1)=\ldots =b_{T}(n_{0}-u)$ and the recurrence relation (\ref{rekurT}b) for $n=n_{0}$ taken modulo $h$ simplifies to:
\begin{displaymath}
0\equiv b_{T}(k_{n_{0}})-b_{T}(k_{n_{0}-1})=\sum_{j=0}^{u}r_{j}b_{T}(n_{0}-j)=rb_{T}(n_{0})\pmod h.
\end{displaymath}
This contradicts the assumption that $r$ and $h$ are co--prime and finishes the proof.
\end{proof}

As a nice consequence of the above theorem we get the following:

\begin{col}\label{ColMain2}
Let $L=(1,-1)$ and $h\in\mathbb{N}$ be co--prime to $r$ and satisfy $2\leq h\leq \degl K-u+1$. If $n<k_{n}$ and $u\leq k_{n}-k_{n-1}-1$ for any $n>0$ and there exist at least one number $n_{0}$ such that $b_{T}(n_{0})$ is co--prime to $h$, then any residue class modulo $h$ appears in the sequence $(b_{T}(n)\mod h)_{n\in\mathbb{N}_{0}}$ infinitely many times. 
\end{col}

\begin{exple}
Any condition of Theorem \ref{ColMain} can be satisfied. For the first one we can take:
\begin{displaymath}
\left\{ 
\begin{array}{lr}
b_{T}(hn)=b_{T}(hn+1)=\ldots =b_{T}(hn+(h-1)), & n\geq 0, \\
b_{T}(hn)-b_{T}(hn-1)=b_{T}(n), & n\geq 0, \\
b_{T}(0)=h,
\end{array} 
\right.
\end{displaymath}
for any $h>1$.

The following example for the second condition is more complicated. Let $\bb_{T}$ be defined in the following way:
\begin{displaymath}
\left\{ 
\begin{array}{lr}
b_{T}(3n)=b_{T}(3n+1)=b_{T}(3n+2), & n\geq 0, \\
b_{T}(3n)-b_{T}(3n-1)=b_{T}(n), & n\geq 0, \\
b_{T}(0)=2.
\end{array} 
\right.
\end{displaymath}
We consider the above sequence modulo $4$. Firstly, observe that $b_{T}(n)$ is even for any $n$ (this follows from the simple induction argument). Hence, $2b_{T}(n)\equiv 0\pmod 4$ for any $n$. Thus,
\begin{displaymath}
\begin{split}
b_{T}(9n)= & b(9n-3)+b_{T}(3n)= b_{T}(9n-6)+b_{T}(3n-3)+b_{T}(3n) \\
 = & b_{T}(9(n-1))+2b_{T}(3n-3)+b_{T}(3n)\equiv b_{T}(9(n-1))+b_{T}(3n)\pmod 4.
\end{split}
\end{displaymath}
This implies that if $b_{T}(3n)\equiv 2\pmod 4$, then $b_{T}(9n)\equiv 2\pmod 4$ or $b_{T}(9(n-1))\equiv 2\pmod 4$. Moreover, if $n\geq 2$, then $9n>9(n-1)>3n$. Thus, the congruence $b_{T}(6)\equiv 2\pmod 4$ (this can be checked by hand) implies the existence of infinitely many numbers $n$ for which $b_{T}(n)\equiv 2\pmod 4$ and hence, condition 2) of Theorem \ref{ColMain} has to be satisfied.

The last condition is satisfied for example by the sequence of $m$--ary partitions $\bb_{m}$ with $h=m$ (this follows from the characterisation modulo $m$).
\end{exple}

We want to use Corollary \ref{ColMain2} for the sequences $\bb_{m}$, $\cc_{m}$ and $\ovbb_{m}$. We have $b_{m}(1)=c_{m}(1)=\ovb_{m}(1)=1$ which is co--prime to any natural number $h$ and hence, we get:

\begin{col}
Let $m\geq 2$ and $h$ be a natural number greater than $3$. If $h\leq m+1$ then in both sequences $(b_{m}(n)\mod h)_{n\in\mathbb{N}_{0}}$ and $(c_{m}(n)\mod h)_{n\in\mathbb{N}_{0}}$ any residue class modulo $h$ appears infinitely many times. If $h\leq m$ is odd, then the same holds for the sequence $(\ovb_{m}(n)\mod h)_{n\in\mathbb{N}_{0}}$.
\end{col}

We use Theorem \ref{TwLd} for the sequences $\bb_{m}$, $\cc_{m}$ and $\ovbb_{m}$ in order to get the next fact.

\begin{col}\label{colb}
Let $m\geq 2$ and $p$ be a prime number greater than $3$.
\begin{enumerate}
\item[\textsc{1)}] If $p\leq m^{2}+m+1$ and there exists $n_{1}$ such that $b_{m}(n_{1})\equiv 0\pmod p$ then the congruence $b_{m}(n)\equiv 0\pmod p$ has infinitely many solutions for $n\in\mathbb{N}$.
\item[\textsc{2)}] If $p\leq m^{2}+m+1$ and there exists $n_{1}$ such that $c_{m}(n_{1})\equiv 0\pmod p$ then the congruence $c_{m}(n)\equiv 0\pmod p$ has infinitely many solutions for $n\in\mathbb{N}$.
\item[\textsc{3)}] If $p\leq m^{2}+m$ and there exists $n_{1}$ such that $\ovb_{m}(n_{1})\equiv 0\pmod p$ then the congruence $\ovb_{m}(n)\equiv 0\pmod p$ has infinitely many solutions for $n\in\mathbb{N}$.
\end{enumerate}
\end{col}

We checked the existence of solutions of the congruences $b_{m}(n)\equiv 0\pmod p$ and $c_{m}(n)\equiv 0\pmod p$ for any $3\leq m\leq 328$ and any prime number $m+2\leq p\leq m^{2}+m+1$. We found the solutions in all of these cases and added the list of the smallest solutions for $3\leq m\leq 12$ in Appendix A and Appendix B. 
As a consequence of our computations and Corollary \ref{colb}, we get that the congruences $b_{m}(n)\equiv 0\pmod p$ and $c_{m}(n)\equiv 0\pmod p$ have infinitely many solutions for primes in the considered range. The links to the full lists of solutions are given in the Appendixes.

As a simple consequence of Theorem \ref{Twp3} we obtain the next facts.

\begin{col}
If $\degu K\geq u+1$ then there is infinitely many $n$ such that $b_{T}(n)$ is divisible by $3$.
\end{col}

\begin{proof}
By the definition of the upper degree there exist infinitely many natural numbers $n$ for which $k_{n+1}-k_{n}\geq u+1$ and the result easily follows from Theorem \ref{Twp3}.
\end{proof}

\begin{col}
Let $m\geq 2$. Then, for any $n\in\mathbb{N}$ in any of the following cases at least one of the numbers:
\begin{enumerate}
\item[\textsc{1)}] $b_{m}(mn)$, $b_{m}(m^{2}n-m)$, $b_{m}(m^{2}n)$ and $b_{m}(m^{2}n+m)$,
\item[\textsc{2)}] $c_{m}(mn+1)$, $c_{m}(m^{2}n+1)$, $c_{m}(m^{2}n+m+1)$ and $c_{m}(m^{2}n+2m+1)$,
\item[\textsc{3)}]$\ovb_{m}(mn)$, $\ovb_{m}(m^{2}n)$, $\ovb_{m}(m^{2}n+m)$ and $\ovb_{m}(m^{2}n+2m)$
\end{enumerate}
is divisible by $3$.
\end{col}

\begin{proof}
By Theorem \ref{Twp3} it is enough to observe that for any natural number $n$ we have $k_{n+1}-k_{n}=m\geq 2$.
\end{proof}

\section{The level and the rank}

We would like to have some information about the values $n$ for which $b_{T}(n)$ is divisible by some fixed number $h$. In this section we focus on finding sets containing at least one such value.

We again narrow down to the triple $T=(K,L,R)$ with $L=(1,-1)$ and $R=(1)$. Similar constructions can be done without any assumptions on $L$ and $R$ if we consider only $h$ not dividing $r$. However, in this generalization the construction becomes very complicated. 

Firstly, we want to exhibit the main idea. Let us take a look at the proof of Theorem \ref{Twp3} in the case $\deg K=2$. Without loss of generality we can assume $k_{n}=2n$ for any $n\in\mathbb{N}_{0}$. The proof is as follows: we take some natural number $n$ and values $c:=b_{T}(2n)$ and $d:=b_{T}(4n-2)$. Then, using the recurrence formula for $\bb_{T}$, we get equations for $b_{T}(4n)$ and $b_{T}(4n+2)$ involving $c$ and $d$, which allow us to say that at least one of these numbers is divisible by~$3$. 

We can repeat this construction for any $h$, but now we need to have more information to say something about divisibility by $h$ in the sequence $\bb_{T}$. That is, we need to produce more equations. Because of this, we take a number $e:=b_{T}(8n-6)$. By (\ref{rekurT}) (in this case this looks similar to (\ref{rekur})) we obtain the following set of equations:
\begin{equation}\label{equ1}
\begin{array}{lll}
b_{T}(2n)=c,\ \ \ \ \ \ \ \ \ \ & b_{T}(4n-2)=d,\ \ \ \ \ \ \ \ \ \ \ \ \ \ \ \ \ \ & b_{T}(8n-6)=e, \\
 & \ \ \ \ \ b_{T}(4n)=d+c, & b_{T}(8n-4)=e+d, \\
 & b_{T}(4n+2)=d+2c, & b_{T}(8n-2)=e+2d, \\
 & & \ \ \ \ \ b_{T}(8n)=e+3d+c, \\
 & & b_{T}(8n+2)=e+4d+2c, \\
 & & b_{T}(8n+4)=e+5d+4c, \\
 & & b_{T}(8n+6)=e+6d+6c. \\
\end{array}
\end{equation}
Now we threat $c$, $d$ and $e$ as variables and check whether that can take values such that none of the numbers above is equal to $0$ modulo some fixed $h$, for example $h=5$. Unfortunately, such a choice of the triple $(c,d,e)$ is possible. Indeed, it is enough to take $(c,d,e)$ such that $(c\mod 5,d\mod 5,e\mod 5)=(x,x,2x)$ or $(x,2x,4x)$ for some non--zero $x\in\mathbb{Z}/5\mathbb{Z}$. Thus, we take a number $f:=b_{T}(16n-14)$ and consider numbers of the form $16n+t$ for $t=-14, -12, -10,\ldots , 12, 14$. Then, we get new set of equations consist of these from (\ref{equ1}) and the following new ones:
\begin{align}\label{equ2}
b_{T}(16n-14)= & f, \nonumber \\
b_{T}(16n-12)= & f+e, \nonumber \\
b_{T}(16n-10)= & f+2e, \nonumber \\
b_{T}(16n-8)= & f+3e+d, \nonumber \\
b_{T}(16n-6)= & f+4e+2d, \nonumber \\
b_{T}(16n-4)= & f+5e+4d, \nonumber \\
b_{T}(16n-2)= & f+6e+6d, \\
b_{T}(16n)= & f+7e+9d+c, \nonumber \\
b_{T}(16n+2)= & f+8e+12d+2c, \nonumber \\
b_{T}(16n+4)= & f+9e+16d+4c, \nonumber \\
b_{T}(16n+6)= & f+10e+20d+6c, \nonumber \\
b_{T}(16n+8)= & f+11e+25d+10c, \nonumber \\
b_{T}(16n+10)= & f+12e+30d+14c, \nonumber \\
b_{T}(16n+12)= & f+13e+36d+20c, \nonumber \\
b_{T}(16n+14)= & f+14e+42d+26c. \nonumber 
\end{align}

Hence, we get the next equalities modulo $5$:

\noindent
\begin{minipage}{0.5\textwidth}
\begin{center} 
for $(c,d,e)=(x,2x,4x)$:
\end{center}
\begin{displaymath}
\begin{split}
b_{T}(16n-14)= & f, \\
b_{T}(16n-12)= & f+4x, \\
b_{T}(16n-10)= & f+3x, \\
b_{T}(16n-8)= & f+4x, \\
b_{T}(16n-6)= & f, \\
b_{T}(16n-4)= & f+3x, \\
b_{T}(16n-2)= & f+x, \\
b_{T}(16n)= & f+x, \\
b_{T}(16n+2)= & f+3x, \\
b_{T}(16n+4)= & f+2x, \\
b_{T}(16n+6)= & f+3x, \\
b_{T}(16n+8)= & f+4x, \\
b_{T}(16n+10)= & f+2x, \\
b_{T}(16n+12)= & f+4x, \\
b_{T}(16n+14)= & f+x.
\end{split}
\end{displaymath}
\end{minipage}
\begin{minipage}{0.5\textwidth}
\begin{center}
for $(c,d,e)=(x,x,2x)$:
\end{center}
\begin{displaymath}
\begin{split}
b_{T}(16n-14)= & f, \\
b_{T}(16n-12)= & f+2x, \\
b_{T}(16n-10)= & f+4x, \\
b_{T}(16n-8)= & f+2x, \\
b_{T}(16n-6)= & f, \\
b_{T}(16n-4)= & f+4x, \\
b_{T}(16n-2)= & f+3x, \\
b_{T}(16n)= & f+4x, \\
b_{T}(16n+2)= & f, \\
b_{T}(16n+4)= & f+3x, \\
b_{T}(16n+6)= & f+x, \\
b_{T}(16n+8)= & f+2x, \\
b_{T}(16n+10)= & f+3x, \\
b_{T}(16n+12)= & f+2x, \\
b_{T}(16n+14)= & f+x.
\end{split}
\end{displaymath}
\end{minipage}
\\

Thus, for any choice of $f$ some of the numbers above are equal to $0$ modulo $5$. This proves that there exist infinitely many $n$ such that $b_{T}(n)$ is divisible by $5$. In fact we get more precise information which says that for any $n$ at least one of the numbers $b_{T}(2^{s}n+2t_{s})$, $s\in\{1,\ldots ,4\}$, $t_{s}\in\{-2^{s-1}+1,-2^{s-1}+2,\ldots ,2^{s-1}-1\}$ is divisible by $5$. This was our aim.

Now we want to generalize the above construction. Let us introduce for any natural number $s$, {\it the $s$--th level} of $n$ in the sequence $\bb_{T}$, denoted by $\lvl(n,T,s)$, inductively as follows: let $b_{T}(k_{n})=c_{1}$ for some $c_{1}$ and let $\lvl(n,T,1):=\{c_{1}\}$ be the first level of $n$ in $b_{T}$. Similarly as before we can assume that $b_{T}(k_{k_{n-1}})=c_{2}$ for some $c_{2}$ and using the recurrence relation (\ref{rekurT}) to get formulas for the numbers of the form $b_{T}(k_{k_{n-1}+t})$ for $t=0,\ldots ,k_{n+1}-k_{n-1}-1$. We define the second level of $n$ in $\bb_{T}$ by the set containing the element from the first level and formulas obtained in the procedure described above.

Assume we have defined the $(s-1)$--th level of $n$ in $b_{T}$. Let $n'$ be such that $b_{T}(n')=c_{s-1}$ (in other words, $n'$ is the smallest number for which the formula for $b_{T}(n')$ depending on $c_{1},\ldots ,c_{s-1}$ is in $\lvl(n,T,s-1)\setminus\lvl(n,T,s-2)$ ). Again, we can add a new number $c_{s}:=b_{T}(k_
{n'-1})$ and using the recurrence relation (\ref{rekurT}) and formulas from the set $\lvl(n,T,s-1)\setminus\lvl(n,T,s-2)$ produce all possible new formulas for $b_{T}(n'')$ depending on $c_{1},\ldots ,c_{s}$ (without adding new variables) for corresponding numbers $n''$. We define $\lvl(n,T,s)$ as a set containing these new formulas and all formulas from the set $\lvl(n,T,s-1)$.

\begin{rem}
We will treat the elements of the sets $\lvl(n,K,s)$ as formulas depending on the variables $c_{1},\ldots ,c_{s}$.
\end{rem}

\begin{rem}
If $L=(1,-1)$, $R=(1)$ and $K$ is such that $\deg K$ exists, i.e., the difference $m:=k_{n+1}-k_{n}$ is constant for all but finitely many numbers $n$, we will use the notation $\lvl(n,m,s):=\lvl(n,T,s)$. In this case for any $s\in\mathbb{N}$, the sets $\lvl(n,m,s)$ contain the same formulas for all $n$ big enough.
\end{rem}

\begin{exple}
Let us look at the set of equalities (\ref{equ1}). The first column gives the set $\lvl(n,2,1)$, the first together with the second give $\lvl(n,2,2)$ and all three columns give $\lvl(n,2,3)$. If we add to that set the formulas from (\ref{equ2}) we get $\lvl(n,2,4)$.
\end{exple}

Notice that we have performed the above construction for any sequence $K$ and $L=(1,-1)$, $R=(1)$. If $L$ and $R$ are arbitrary we can do the same construction but then the set $\lvl(n,T,1)$ has to contain $u+1$ elements and we have to add $t$ new variables at each step. 

The next definition is stated for any triple $T$. We define {\it the rank} of a natural number $n$ in the sequence $b_{T}$ modulo natural number $h$ as the least number $\rk(n,T,h)$ such that at least one of the elements of $\lvl(n,T,\rk(n,T,h))$ is divisible by $h$. If we have no information about the numbers $c_{1},c_{2},\ldots$ and treat them as a variables (that is, we assume they can be equal to any number between $0$ and $h-1$), we call the number $\rk(n,T,h)$ the rank of $n$ in $b_{T}$ modulo $h$ in the general case and denote this by $\rk_{g}(n,T,h)$. If there are no such numbers, then we define $\rk(n,T,h)=\infty$ and $\rk_{g}(n,T,h)=\infty$ respectively.

\begin{rem}
If $m=\deg K$ exists we denote $\rk(n,m,h):=\rk(n,T,h)$ and $\rk_{g}(n,m,h):=\rk_{g}(n,T,h)$.
\end{rem}

\begin{exple}
Theorem \ref{Twp3} gives us $\rk_{g}(n,T,3)=2$ for any $n$ for which $k_{n+1}-k_{n}\geq u+1$. On the beginning of this section we also showed that $\rk_{g}(n,2,5)=4$ for any natural number $n$ big enough.
\end{exple}

\begin{exple}
It is possible that $\rk(n,T,h)\neq\rk_{g}(n,T,h)$. For example, consider a triple $T$ with $L=(1,-1)$, $R=(1)$ and a sequence $K$ such that $\deg K=2$. Using computer calculations we checked that in the general case $r_{g}(n,T,4)>10$ for $n$ big enough. Moreover, the computations suggest that $r_{g}(n,T,4)=\infty$ in this case. On the other hand, taking the triple $T$ such that $\bb_{T}=\bb_{2}$ we get $\deg T=2$ and $\rk(n,T,4)\leq 2$ for any $n\geq 2$. However, we always have $\rk(n,T,h)\leq\rk_{g}(n,T,h)$.
\end{exple}

Let 
\begin{displaymath}
S(n,T,s):=\{\ j\in\mathbb{N}_{0}\ |\ b_{T}(j)\in\lvl(n,T,i)\text{ for some } i\leq s\ \}.
\end{displaymath}
Using the above definitions we can formulate the following:

\begin{theorem}
Let $T$ be an arbitrary triple and $h$ be a natural number. If $\rk(n,T,h)<\infty$, then at least one of the numbers $b_{T}(k)$ for $k\in S(n,T,\rk(n,T,h))$ is divisible by $h$.
\end{theorem}

In particular, we get:

\begin{col}
Let $h$, $n$ and $m\geq 2$ be natural numbers. If $\rk(n,m,h)<\infty$, then at least one of the numbers $b_{m}(m^{s}n+mt_{s})$, $s\in\{1,\ldots ,\rk(n,m,h)\}$, $t_{s}\in\{-m^{s-1}+1,-m^{s-1}+2,\ldots ,m^{s-1}-1\}$ is divisible by $h$.
\end{col}

The previous facts, together with the tables of values for $\rk_{g}(n,2,h)$ and $\rk_{g}(n,3,h)$ for $h$ added in Appendix~C, generalize Theorems \ref{comgen}, \ref{comgen2} and \ref{twcom} in the case of $\bb_{T}$ with $\deg K=2$ and $\deg K=3$. We also add a table of upper bounds of $\rk(n,2,h)$ in the case of $\bb_{T}=\bb_{2}$ for numbers $T$ equal $4$ modulo $8$ and for $n\geq 2$. In our computations in this special case we used the fact that $b_{2}(n)$ is even for $n\geq 2$.

While checking the values of $\rk_{g}(n,2,h)$ we noticed interesting phenomenon. We present this in the next question.

\begin{que}
Are the following statements true: 

 $1)$ $\rk_{g}(n,2,4h)=\infty$, 
 
 $2)$ $0\leq\rk_{g}(n,2,4l+3)-\rk_{g}(n,2,4l+1)\leq 1$, 
 
 $3)$ $0\leq\rk_{g}(n,2,4h+2)-\rk_{g}(n,2,4h+1)\leq 1$, 

 $4)$ $\rk_{g}(n,2,4h+1)\leq\rk_{g}(n,2,4h+5)$, 
 
\begin{flushleft}
for  any natural numbers $h$ and $n$ large enough?
\end{flushleft}
\end{que}

According to the first statement of above question, we expect that for the choice $c_{1}=c_{2}=\ldots =1$ none of elements of $\lvl(n,2,r)$ would be divisible by $4$ for any $r$.

\section{Sequences with $L=(1,-1)$ modulo $\deg K$}

Take a sequence $K$ such that $\deg K=m$ exists. In this paragraph we want to give a characterisation modulo $m$ of the recurrence sequences with $L=(1,-1)$. 

We use the notation $b_{T}(q)=0$ for $q\notin \mathbb{N}_{0}$ and $k_{-n}=-n$ for $n\in\mathbb{N}$.

Let us start with a well known fact characterizing the sequence of $m$--ary partitions modulo $m$. The first proof of this fact, given by Alkauskas, is done by simple induction argument. Andrews {\it et. al.} used power series in their proof and 
Edgar's approach was more combinatorial and gave more information. Now, we give a new proof.

\begin{stw}\label{charactb}
Let $m>1$ and $n\in\mathbb{N}_{0}$. If $n=a_{0}+a_{1}m+\ldots +a_{s}m^{s}$ where $0\leq a_{j}\leq m-1$ for $j=0,\ldots ,s$, then 
\begin{displaymath}
b_{m}(mn)\equiv\prod_{j=0}^{s}(a_{j}+1)\pmod m.
\end{displaymath}
\end{stw}
\begin{proof}
From (\ref{rekur}b) we have for any $n\in\mathbb{N}_{0}$:
\begin{displaymath}
b_{m}(mn)=\sum_{i=0}^{n}b_{m}(i).
\end{displaymath}
We can write $n=mn'+a_{0}$ for some $a_{0}\in\{0,\ldots ,m-1\}$ and using (\ref{rekur}a) get:
\begin{displaymath}
\begin{split}
b_{m}(m(mn'+a_{0}))= & \sum_{i=0}^{mn'+a_{0}}b_{m}(i)=(a_{0}+1)b_{m}(mn')+m\left(\sum_{i=0}^{n'-1}b_{m}(mi)\right) \\
\equiv & (a_{0}+1)b_{m}(mn')\pmod m.
\end{split}
\end{displaymath}
Now we can write $n'=mn''+a_{1}$ and repeat the reasoning. After $s+1$ steps we get the result.
\end{proof}

This proof can be adapted to a more general case. We start by the following important

\begin{lem}\label{LemChar}
Let $\deg K=m$ and $n_{0}\in\mathbb{N}_{0}$ be the smallest number such that the difference $k_{n+1}-k_{n}=m$ for any $n\geq n_{0}$. If $u<k_{n+1}-k_{n}$ for any $n\geq n_{0}-1$, then:
\begin{displaymath}
\begin{split}
b_{T}(k_{k_{n}+q})= & b_{T}(k_{k_{n_{0}}-1}) + b_{T}(k_{n_{0}-1})\sum_{i=1}^{u}ir_{i} + b_{T}(k_{n})\sum_{i=0}^{q}(q-i+1)r_{i} \\
 &  + b_{T}(k_{n-1})\sum_{i=q+1}^{u}(m+q-i+1)r_{i} + m\left(\sum_{i=0}^{u}r_{i}\right)\left(\sum_{j=n_{0}}^{n-1}b_{T}(k_{j})\right)
\end{split}
\end{displaymath}
for any natural numbers $n\geq n_{0}$ and $q\in\{0,\ldots ,m-1\}$.
\end{lem}
\begin{proof}
We have $b_{T}(k_{n})-b_{T}(k_{n-1})=\sum_{i=0}^{u}r_{i}b_{T}(n-i)$. Denote the right side of this equality by $a(n)$. Thus, for $n\geq k_{n_{0}}-1$, we have the equality:
\begin{displaymath}
b_{T}(k_{n})=b_{T}(k_{k_{n_{0}}-1})+\sum_{j=k_{n_{0}}}^{n}a(j).
\end{displaymath}
Similarly as before, we put $k_n+q$ instead of $n$ and get:
\begin{align*}
b_{T}(k_{k_{n}+q})= & b_{T}(k_{k_{n_{0}}-1})+\sum_{j=k_{n_{0}}}^{k_{n}+q}a(j) \\
 = & b_{T}(k_{k_{n_{0}}-1}) + \sum_{j=k_{n_{0}}}^{k_{n}+q}\sum_{i=0}^{u}r_{i}b_{T}(j-i) \\
 = & b_{T}(k_{k_{n_{0}}-1}) + \sum_{i=0}^{u}r_{i}\left(\sum_{j=k_{n_{0}}}^{k_{n}+q}b_{T}(j-i)\right) \\
 = & b_{T}(k_{k_{n_{0}}-1}) + \sum_{i=0}^{u}r_{i}\left(\sum_{j=k_{n_{0}}-i}^{k_{n}-i+q}b_{T}(j)\right) \\
 = & b_{T}(k_{k_{n_{0}}-1}) + \sum_{i=0}^{u}r_{i}\left(\sum_{j=k_{n_{0}}}^{k_{n}-i+q}b_{T}(j)\right) + \sum_{i=1}^{u}ir_{i}b_{T}(k_{n_{0}-1}) \\
 = & b_{T}(k_{k_{n_{0}}-1}) + b_{T}(k_{n_{0}-1})\sum_{i=1}^{u}ir_{i} + \sum_{i=0}^{q}r_{i}\left(\sum_{j=k_{n}}^{k_{n}+q-i}b_{T}(k_{n}) + m\left(\sum_{j=n_{0}}^{n-1}b_{T}(k_{j})\right)\right) \\
 & + \sum_{i=q+1}^{u}r_{i}\left(\sum_{j=k_{n-1}}^{k_{n}+q-i}b_{T}(k_{n-1}) + m\left(\sum_{j=n_{0}}^{n-1}b_{T}(k_{j})\right)\right) %+ \sum_{i=2}^{u}r_{i}b(k_{n_{0}-1}) 
 \\
 = & b_{T}(k_{k_{n_{0}}-1}) + b_{T}(k_{n_{0}-1})\sum_{i=1}^{u}ir_{i} \\ & + \sum_{i=0}^{q}r_{i}\left(\sum_{j=k_{n}}^{k_{n}+q-i}b_{T}(k_{n})\right)+ m\left(\sum_{i=0}^{q}r_{i}\right)\left(\sum_{j=n_{0}}^{n-1}b_{T}(k_{j})\right) \\
   & + \sum_{i=q+1}^{u}r_{i}\left(\sum_{j=k_{n-1}}^{k_{n-1}+m-(i-q)}b_{T}(k_{n-1})\right) +m\left(\sum_{i=q+1}^{u}r_{i}\right)\left(\sum_{j=n_{0}}^{n-1}b_{T}(k_{j})\right) \\
 = & b_{T}(k_{k_{n_{0}}-1}) + b_{T}(k_{n_{0}-1})\sum_{i=1}^{u}ir_{i} + b_{T}(k_{n})\sum_{i=0}^{q}(q-i+1)r_{i} \\ & + b_{T}(k_{n-1})\sum_{i=q+1}^{u}(m+q-i+1)r_{i} + m\left(\sum_{i=0}^{u}r_{i}\right)\left(\sum_{j=n_{0}}^{n-1}b_{T}(k_{j})\right).
\end{align*}
\end{proof}

The next fact is a simple consequence of Lemma \ref{LemChar}.

\begin{theorem}\label{charact}
Let $\deg K=m$ and $n_{0}\in\mathbb{N}_{0}$ be the smallest number such that the difference $k_{n+1}-k_{n}=m$ for any $n\geq n_{0}$. If $u<k_{n+1}-k_{n}$ for any $n\geq n_{0}-1$, then:
\begin{displaymath}
\begin{split}
b_{T}(k_{k_{n}+q})\equiv & b_{T}(k_{k_{n_{0}}-1}) + b_{T}(k_{n_{0}-1})\sum_{i=1}^{u}ir_{i} \\
 & + b_{T}(k_{n})\sum_{i=0}^{q}(q-i+1)r_{i} + b_{T}(k_{n-1})\sum_{i=q+1}^{u}(q-i+1)r_{i}\pmod m
\end{split}
\end{displaymath}
for any natural numbers $n\geq n_{0}$ and $q\in\{0,\ldots ,m-1\}$.
\end{theorem}
\begin{proof}
We take the identity from Lemma \ref{LemChar} modulo $m$.
\end{proof}
Now we can write $n$ and $n-1$ in the form $k_{n'}+q'$ and repeat the above computations. We can do this as long as  the numbers $n$ and $n-1$ are greater than $k_{n_{0}}+1$. Hence, on the end we get the formula for $b_{T}(n)\mod m$ depending on coefficients $q_{0},\ldots , q_{t}$ in representation $n=k_{k_{\ldots _{k_{n_{0}}+q_{0}}}+q_{t-1}}+q_{t}$. % On the other hand, we know that any number $k_{n}$ for $n\geq n_{0}$ is of the form $mn+j$ for some $j\in\{0,\ldots ,m\}$ which does not depend on $n$. Thus we can change the formula for $b_{T}(n)\mod m$ into formula depending on the base--$m$ representation of number $n$.

In the case of sequences $\cc_{m}$ and $\ovbb_{m}$ we get:

\begin{theorem}\label{charactcovb}
For $m\geq 2$ we have:
\begin{enumerate}
\item[\textsc{1)}] $c_{m}(m(mn+q+1)+1)\equiv 1+(q+1)c_{m}(mn+1)\pmod m$,
\item[\textsc{2)}] $\ovb_{m}(m(mn+q))\equiv (2q+1)\ovb_{m}(mn)\pmod m$
\end{enumerate}
for any $n\geq 1$ and $q\in\{0,\ldots ,m-1\}$.
\end{theorem}
\begin{proof}
1) By Example \ref{exple1} we know that in this case $K$ is defined by $k_{0}=0$ and $k_{n}=mn+1$ for $n\geq 1$. Moreover, $L=(1,-1)$ and $R=(1)$ and thus $u=0$, $n_{0}=1$, $k_{n_{0}-1}=k_{0}=0$ and $k_{k_{n_{0}}-1}=k_{m}$. Hence:
\begin{displaymath}
\begin{split}
 c_{m}(k_{k_{n_{0}}-1})= & c_{m}(k_{m})=\sum_{i=0}^{m}c_{m}(i) \\ 
  = & (m+1)c_{m}(0)=m+1\equiv 1\pmod m, \\
 c_{m}(k_{n_{0}-1})\sum_{i=1}^{u}ir_{i}= & 0, \\
 c_{m}(k_{n})\sum_{i=0}^{q}(q-i+1)r_{i}= & c_{m}(mn+1)(q+1), \\
 c_{m}(k_{n-1})\sum_{i=q+1}^{u}(q-i+1)r_{i}= & 0.
\end{split}
\end{displaymath}
Taking the above together and using Theorem \ref{charact} we get the first part of the statement of the theorem. 

2) By Example \ref{exple1} we have $K=(mn)_{n=0}^{\infty}$, $L=(1,-1)$ and $R=(1,1)$. Hence, $n_{0}=0$, $k_{k_{n_{0}}-1}=k_{n_{0}-1}=-1$, $u=1$, $r_{0}=r_{1}=1$. By Theorem \ref{charact} we get:
\begin{displaymath}\begin{split}
\ovb_{m}(m(mn+q))\equiv &  
\left\{
\begin{array}{ll}
 \ovb_{m}(mn)(q+1)+\ovb_{m}(m(n-1))q, & q=0 \\
 \ovb_{m}(mn)(q+1+q), & q>0
\end{array}
\right. \\
 = & 
 \left\{
\begin{array}{ll}
 \ovb_{m}(mn), & q=0 \\
 (2q+1)\ovb_{m}(mn), & q>0
\end{array}
\right. \\
 = & (2q+1)\ovb_{m}(mn)\pmod m
\end{split}\end{displaymath}
and the result follows.
\end{proof}

Using the induction argument similar to the one from the end of the proof of Proposition \ref{charactb}, we get for the sequence $\cc_{m}$ the known formula found by Andrews {\it et. al.} (the equivalence of the formula presented in the paper of Andrews {\it et. al.} and the formula from Theorem \ref{charactcovb} was pointed out also by Ekhad and Zeilberger in \cite{Ekh}). Similarly, we get the following characterisation for the sequence~$\ovbb_{m}$.

\begin{col}\label{charactovbcol}
Let $m>1$ and $n\in\mathbb{N}_{0}$. If $n=a_{0}+a_{1}m+\ldots +a_{s}m^{s}$ where $0\leq a_{j}\leq m-1$ for $j=0,\ldots ,s$, then 
\begin{displaymath}
\ovb_{m}(mn)\equiv\prod_{j=0}^{s}(2a_{j}+1)\pmod m.
\end{displaymath}
\end{col}

\section{Sequences with $L=(1,-1)$ modulo powers of $\deg K$}

In this section we want to present a general method which can lead to the full characterisation of sequences from the previous chapter modulo any power of $m$.

For any $\alpha >\beta$ we use the convention:
\begin{displaymath}
\begin{split}
\sum_{i=\alpha}^{\beta}\text{something}(i)= & 0, \\
\prod_{i=\alpha}^{\beta}\text{something}(i)= & 
\left\{
\begin{array}{ll}
1, & \text{if }\alpha -\beta =1, \\
0, & \text{if }\alpha -\beta >1.
\end{array}
\right.
\end{split}
\end{displaymath}

Recall that in Lemma \ref{LemChar} we have shown the following identity:
\begin{equation}\label{identityb}
\begin{split}
b_{T}(k_{k_{n}+q})= & b_{T}(k_{k_{n_{0}}-1}) + b_{T}(k_{n_{0}-1})\sum_{i=1}^{u}ir_{i} + b_{T}(k_{n})\sum_{i=0}^{q}(q-i+1)r_{i} \\
 &  + b_{T}(k_{n-1})\sum_{i=q+1}^{u}(m+q-i+1)r_{i} + m\left(\sum_{i=0}^{u}r_{i}\right)\left(\sum_{j=n_{0}}^{n-1}b_{T}(k_{j})\right).
\end{split}
\end{equation}
By Theorem \ref{charact} we get the characterisation of $b_{T}(n)\mod m$ depending on base--$m$ representation of $n$ and thus we can use it to compute the sum $\sum_{j=n_{0}}^{n-1}b_{T}(k_{j})$ modulo $m$. Having this we can consider the equality (\ref{identityb}) modulo $m^{2}$. Then we can similarly use this characterisation modulo $m^{2}$ for the sum $\sum_{j=n_{0}}^{n-1}b_{T}(k_{j})$ and consider (\ref{identityb}) modulo $m^{3}$. Repeating this process we can get a characterisation of the behaviour of our sequence modulo any power of $m$. However, the computations are very difficult even in the easiest cases so we narrow our considerations down to the sequence of $m$--ary partitions modulo $\mu_{2}$ where:
\begin{displaymath}
\mu_{2}:=
\left\{ 
\begin{array}{ll}
m^{2} & \text{for $m$ odd}, \\
\frac{m^{2}}{2} & \text{for $m$ even}.
\end{array} 
\right.
\end{displaymath}

We have the following:

\begin{theorem}\label{TwCharpowersbm}
Let $m>1$ and $n\in\mathbb{N}_{0}$. If $n=a_{0}+a_{1}m+\ldots +a_{s}m^{s}$, where $0\leq a_{j}\leq m-1$ for $j=0,\ldots ,s$, then 
\begin{displaymath}
b_{m}(mn)\equiv\prod_{i=0}^{s}(a_{i}+1)+m\bigg(\sum_{i=1}^{s}\frac{a_{i}(a_{i}+1)}{2}\prod_{\substack{j=0, \\ j\neq i,i-1}}^{s}(a_{j}+1)\bigg)\pmod{\mu_{2}}.
\end{displaymath}
\end{theorem}
\begin{proof}
Write $n=a_{0}+a_{1}m+\ldots +a_{s}m^{s}$. The identity (\ref{identityb}) implies:
\begin{align*}
b_{m}(mn)= & b_{m}(a_{0}m+a_{1}m^{2}+\ldots +a_{s}m^{s+1}) \\ 
 = & (a_{0}+1)b_{m}(a_{1}m+a_{2}m^{2}\ldots +a_{s}m^{s})+m\sum_{i=0}^{a_{1}+a_{2}m\ldots +a_{s}m^{s-1}-1}b_{m}(mi) \\
 = & (a_{0}+1)(a_{1}+1)b_{m}(a_{2}m+a_{3}m^{2}\ldots +a_{s}m^{s-1}) \\
   & +m\bigg(\sum_{i=0}^{a_{1}+a_{2}m\ldots +a_{s}m^{s-1}-1}b_{m}(mi)+(a_{0}+1)\sum_{i=0}^{a_{2}+a_{3}m\ldots +a_{s}m^{s-2}-1}b_{m}(mi)\bigg) \\
 = & \ldots =(a_{0}+1)(a_{1}+1)\ldots (a_{s}+1) \\
   & +m\bigg(\sum_{i=0}^{a_{1}+a_{2}m\ldots +a_{s}m^{s-1}-1}b_{m}(mi)+(a_{0}+1)\sum_{i=0}^{a_{2}+a_{3}m\ldots +a_{s}m^{s-2}-1}b_{m}(mi) \\
   & +\ldots +\prod_{j=0}^{s-3}(a_{j}+1)\sum_{i=0}^{a_{s-1}+a_{s}m-1}b_{m}(mi) +\prod_{j=0}^{r-2}(a_{j}+1)\sum_{i=0}^{a_{s}-1}b_{m}(mi)\bigg).
\end{align*}
By Theorem \ref{charact} we get for any natural number $h$:
\begin{displaymath}
\sum_{i=mh}^{mh+m-1}b_{m}(mi) \equiv \sum_{i=0}^{m-1}(i+1)b_{m}(mh) = \frac{m(m+1)}{2}b_{m}(mh) \equiv 0\pmod{\mu_{2}}.
\end{displaymath}
Hence, Proposition \ref{charactb} implies for any $t=1,\ldots ,s$ the following congruences are true:
\begin{displaymath}
\begin{split}
\sum_{i=0}^{a_{t}+\ldots +a_{s}m^{s-t}-1}b_{m}(mi) 
 = & \sum_{i=a_{t+1}m\ldots +a_{s}m^{s-t}}^{a_{t}-1+a_{t+1}m+\ldots +a_{s}m^{s-t}}b_{m}(mi)+\sum_{h=0}^{a_{t+1}+\ldots +a_{s}m^{s-t-1}}\sum_{i=mh}^{mh+m-1}b_{m}(mi) \\
 \equiv & \sum_{i=0}^{a_{t}-1}(i+1)(a_{t+1}+1)\ldots (a_{s}+1)\\
 \equiv & \frac{a_{t}(a_{t}+1)}{2}(a_{t+1}+1)\ldots (a_{s}+1)\pmod{\mu_{2}}.
\end{split}
\end{displaymath}
Taking everything together:
\begin{displaymath}
\begin{split}
b_{m}(a_{0}m+a_{1} & m^{2}+\ldots +a_{s}m^{s+1})  \equiv(a_{0}+1)(a_{1}+1)\ldots (a_{s}+1) \\
   & +m\bigg(\frac{a_{1}(a_{1}+1)}{2}(a_{2}+1)\ldots (a_{s}+1)+(a_{0}+1)\frac{a_{2}(a_{2}+1)}{2}(a_{3}+1)\ldots (a_{s}+1) \\
   & +\ldots +(a_{0}+1)\ldots (a_{s-3}+1)\frac{a_{s-1}(a_{s-1}+1)}{2}(a_{s}+1) \\ &+(a_{0}+1)\ldots (a_{s-2}+1)\frac{a_{s}(a_{s}+1)}{2}\bigg)\pmod{\mu_{2}}
\end{split}
\end{displaymath}
and the result follows.
\end{proof}

In the case of the sequence $\ovbb_{m}$ we give only a partial result concerning numbers $\ovb_{m}(n)\mod{m^{2}}$ for $n$ which have no zeros in their base--$m$ representation.

\begin{theorem}\label{TwCharpowersovbm}
Let $m>1$ and $n\in\mathbb{N}_{0}$. If $n=a_{0}+a_{1}m+\ldots +a_{s}m^{s}$ where $1\leq a_{j}\leq m-1$ for $j=0,\ldots ,s$, then 
\begin{displaymath}
\ovb_{m}(mn)\equiv \prod_{i=0}^{s}(2a_{i}+1)+2m\bigg(\sum_{i=1}^{s}a_{i}^{2}\prod_{\substack{j=0, \\ j\neq i,i-1}}^{s}(2a_{j}+1)\bigg)\pmod{m^{2}}.
\end{displaymath}
\end{theorem}
\begin{proof}
The equality (\ref{identityb}) implies:
\begin{align*}
\ovb_{m}(m(a_{0}+\ldots  & +a_{s}m^{s}))\\
 = & (2a_{0}+1)\ovb_{m}(a_{1}m+a_{2}m^{2}\ldots +a_{s}m^{s})+2m\sum_{i=0}^{a_{1}+a_{2}m+\ldots +a_{s}m^{s-1}-1}\ovb_{m}(mi) \\
 = & (2a_{0}+1)(2a_{1}+1)\ovb_{m}(a_{2}m+a_{3}m^{2}+\ldots +a_{s}m^{s-1}) \\
   & +2m\bigg(\sum_{i=0}^{a_{1}+a_{2}m\ldots +a_{s}m^{s-1}-1}\ovb_{m}(mi)+(2a_{0}+1)\sum_{i=0}^{a_{2}+a_{3}m\ldots +a_{s}m^{s-2}-1}\ovb_{m}(mi)\bigg) \\
 = & \ldots =(2a_{0}+1)(2a_{1}+1)\ldots (2a_{s}+1) \\
   & +2m\bigg(\sum_{i=0}^{a_{1}+a_{2}m\ldots +a_{s}m^{s-1}-1}\ovb_{m}(mi)+(2a_{0}+1)\sum_{i=0}^{a_{2}+a_{3}m+\ldots +a_{s}m^{s-2}-1}\ovb_{m}(mi) \\
   & +\ldots +\prod_{j=0}^{s-3}(2a_{j}+1)\sum_{i=0}^{a_{s-1}+a_{s}m-1}\ovb_{m}(mi)+\prod_{j=0}^{s-2}(2a_{j}+1)\sum_{i=0}^{a_{s}-1}\ovb_{m}(mi)\bigg).
\end{align*}
By Theorem \ref{charact} we get for any natural number $h$:
\begin{displaymath}
\sum_{i=mh}^{mh+m-1}\ovb_{m}(mi) \equiv \sum_{i=0}^{m-1}(2i+1)\ovb_{m}(mh) = m^{2}\ovb_{m}(mh) \equiv 0\pmod{m}.
\end{displaymath}
Hence, Theorem \ref{charactcovb} implies for any $t=1,\ldots ,s$:
\begin{displaymath}
\begin{split}
\sum_{i=0}^{a_{t}+\ldots +a_{s}m^{s-t}-1}\ovb_{m}(mi) 
 = & \sum_{i=a_{t+1}m+\ldots +a_{s}m^{s-t}}^{a_{t}-1+a_{t+1}m+\ldots +a_{s}m^{s-t}}\ovb_{m}(mi)+\sum_{h=0}^{a_{t+1}+\ldots +a_{s}m^{s-t-1}}\sum_{i=mh}^{mh+m-1}\ovb_{m}(mi) \\
 \equiv & \sum_{i=0}^{a_{t}-1}(2i+1)(2a_{t+1}+1)\ldots (2a_{s}+1)\\
 \equiv & a_{t}^{2}(2a_{t+1}+1)\ldots (2a_{s}+1)\pmod{m}.
\end{split}
\end{displaymath}
Taking everything together:
\begin{displaymath}
\begin{split}
\ovb_{m}( & a_{0}m+a_{1}m^{2}+\ldots +a_{s}m^{s+1})\\
 = & \ldots =(2a_{0}+1)(2a_{1}+1)\ldots (2a_{s}+1) \\
   & +2m\bigg(\sum_{i=0}^{a_{1}+a_{2}m\ldots +a_{s}m^{s-1}-1}\ovb_{m}(mi)+(2a_{0}+1)\sum_{i=0}^{a_{2}+a_{3}m\ldots +a_{s}m^{s-2}-1}\ovb_{m}(mi) \\
   & +\ldots +\prod_{j=0}^{s-3}(2a_{j}+1)\sum_{i=0}^{a_{s-1}+a_{s}m-1}\ovb_{m}(mi)+\prod_{j=0}^{s-2}(2a_{j}+1)\sum_{i=0}^{a_{s}-1}\ovb_{m}(mi)\bigg) \\
 \equiv & (2a_{0}+1)(2a_{1}+1)\ldots (2a_{s}+1)+2m\bigg(\sum_{i=1}^{s}a_{i}^{2}\prod_{\substack{j=0, \\ j\neq i,i-1}}^{s}(2a_{j}+1)\bigg) \pmod{m^{2}}.
\end{split}
\end{displaymath}
The result follows.
\end{proof}

The problems with numbers $n$ having at least one digit equals to zero in the base--$m$ representation follow from the fact that for them the formula (\ref{identityb}) looks differently than for other numbers.

As a simple corollary from Theorem \ref{TwCharpowersovbm} we get a stronger version of Corollary \ref{charactovbcol} for even $m$ and numbers $n$ which have no zeros in base--$m$ representation.

\begin{col}
Let $m>1$ be even and $n\in\mathbb{N}_{0}$. If $n=a_{0}+a_{1}m+\ldots +a_{s}m^{s}$ where $1\leq a_{j}\leq m-1$ for $j=0,\ldots ,s$, then 
\begin{displaymath}
\ovb_{m}(mn)\equiv\prod_{j=0}^{s}(2a_{j}+1)\pmod{2m}.
\end{displaymath}
\end{col}
\begin{proof}
If $m$ is even, then $2m$ divides $m^{2}$ and the result follows from Theorem \ref{TwCharpowersovbm}.
\end{proof}

\section{Problems and conjectures}

In this section we want to present some open problems and state some conjectures. We believe that their solutions allow us to better understand congruence properties of recurrence sequences connected with partition functions of the considered types.

The first two problems are strictly connected with the first result of this paper.

\begin{pbm}
Is it true that at most one of the conditions from {\rm Theorem \ref{TwMain} } can be satisfied? 
\end{pbm}

\begin{rem}
We showed that this is true if $L=(1,-1)$ and $d$ is given by the special formula (\ref{seqd}).
\end{rem}

\begin{pbm}
Can we skip the assumption $h\leq\degl K-u+1$ in {\rm Theorem \ref{TwMain} }? 
\end{pbm}

\begin{rem}
The computations from the previous Section and Theorem 5 and related results may suggest that this is possible at least in some special cases.
\end{rem}

Alkauskas in his paper gave the following conjecture.

\begin{con}{\sc (Alkauskas)}\label{con1}
For any $h\in\mathbb{N}$, $8\nmid h$, there exist infinitely many $n$ such that $b_{2}(n)$ is divisible by $h$.
\end{con}

In regard to the previous deliberations we can extend this as follows.

\begin{con}
Let $L=(1,-1)$ and $R=(1)$. If $\degl K\geq 2$ then for any $h\in\mathbb{N}$, $4\nmid h$, there exist infinitely many $n$ such that $b_{T}(n)$ is divisible by $l$.
\end{con}

Observe that if we assumed more, i.e., $\degl K\geq 3$, we would get much more elements in the set $\lvl(n,K,h)$ then in the case $\degl K=2$ for any numbers $h$ and $n\gg 1$. We believe that this is enough to omit values of $h$ divisible by $4$. Hence, we state the next conjecture.

\begin{con}\label{Con}
If $\degl K\geq 3$, then for any $h\in\mathbb{N}$ there exist infinitely many $n$ such that $b_{T}(n)$ is divisible by $h$.
\end{con}

Here is another conjecture which is connected with Theorem \ref{TwLd}.

\begin{con}
Let $L=(1,-1)$. If $\degl K\geq 2$ and if there exists $n_{1}$ such that $b_{T}(n_{1})$ is divisible by $h$, then there exist infinitely many numbers $n$ with this property.
\end{con}

On the end we give the following problem:

\begin{pbm}
How may results of {\rm Sections 5} and {\rm 6} be extended to a more general case, that is to the sequences not necessarily satisfying $L=(-1,1)$?
\end{pbm}

\section*{Acknowledgements}

I wish to thank my wife for her patience. I would also like to thank Paulina Pełszyńska for her help with preparing the English version of this paper and Bartłomiej Puget for his help with computations. I am also very grateful to prof. Maciej Ulas for his profound comments.

\newpage

\section*{Appendix A}

\addcontentsline{toc}{section}{Appendix A}

The smallest solution $n_{1}$ of the equation $b_{m}(n)\equiv 0\pmod p$ for $3\leq m\leq 12$ and prime numbers $m+2\leq p\leq m^{2}+m+1$.

\scriptsize{
\begin{center}
\begin{minipage}[t]{.24\textwidth}
\begin{displaymath}
\begin{array}{|c|c|c|}
\hline
m & p & n_{1} \\ \hline
\hline
3 & 5 & 15 \\ \hline
3 & 7 & 21 \\ \hline
3 & 11 & 33 \\ \hline
3 & 13 & 39 \\ \hline
\hline
4 & 7 & 28 \\ \hline
4 & 11 & 44 \\ \hline
4 & 13 & 52 \\ \hline
4 & 17 & 452 \\ \hline
4 & 19 & 68 \\ \hline
\hline
5 & 7 & 25 \\ \hline
5 & 11 & 35 \\ \hline
5 & 13 & 40 \\ \hline
5 & 17 & 75 \\ \hline
5 & 19 & 80 \\ \hline
5 & 23 & 90 \\ \hline
5 & 29 & 225 \\ \hline
5 & 31 & 170 \\ \hline
\hline
6 & 11 & 42 \\ \hline
6 & 13 & 48 \\ \hline
6 & 17 & 90 \\ \hline
6 & 19 & 96 \\ \hline
6 & 23 & 108 \\ \hline
6 & 29 & 270 \\ \hline
6 & 31 & 204 \\ \hline
6 & 37 & 576 \\ \hline
6 & 41 & 150 \\ \hline
6 & 43 & 426 \\ \hline
\hline
7 & 11 & 49 \\ \hline
7 & 13 & 56 \\ \hline
7 & 17 & 105 \\ \hline
7 & 19 & 112 \\ \hline
7 & 23 & 126 \\ \hline
7 & 29 & 315 \\ \hline
7 & 31 & 238 \\ \hline
7 & 37 & 672 \\ \hline
7 & 41 & 175 \\ \hline
7 & 43 & 497 \\ \hline
7 & 47 & 259 \\ \hline
7 & 53 & 581 \\ \hline
\hline
8 & 11 & 56 \\ \hline
\end{array}
\end{displaymath}
\end{minipage}
\begin{minipage}[t]{.24\textwidth}
\begin{displaymath}
\begin{array}{|c|c|c|}
\hline
m & p & n_{1} \\ \hline
\hline
8 & 13 & 64 \\ \hline
8 & 17 & 120 \\ \hline
8 & 19 & 128 \\ \hline
8 & 23 & 144 \\ \hline
8 & 29 & 360 \\ \hline
8 & 31 & 272 \\ \hline
8 & 37 & 768 \\ \hline
8 & 41 & 200 \\ \hline
8 & 43 & 568 \\ \hline
8 & 47 & 296 \\ \hline
8 & 53 & 664 \\ \hline
8 & 59 & 288 \\ \hline
8 & 61 & 368 \\ \hline
8 & 67 & 384 \\ \hline
8 & 71 & 744 \\ \hline
8 & 73 & 264 \\ \hline
\hline
9 & 11 & 63 \\ \hline
9 & 13 & 72 \\ \hline
9 & 17 & 135 \\ \hline
9 & 19 & 144 \\ \hline
9 & 23 & 162 \\ \hline
9 & 29 & 405 \\ \hline
9 & 31 & 306 \\ \hline
9 & 37 & 864 \\ \hline
9 & 41 & 225 \\ \hline
9 & 43 & 639 \\ \hline
9 & 47 & 333 \\ \hline
9 & 53 & 747 \\ \hline
9 & 59 & 324 \\ \hline
9 & 61 & 414 \\ \hline
9 & 67 & 432 \\ \hline
9 & 71 & 837 \\ \hline
9 & 73 & 297 \\ \hline
9 & 79 & 1143 \\ \hline
9 & 83 & 315 \\ \hline
9 & 89 & 234 \\ \hline
\hline
10 & 13 & 80 \\ \hline
10 & 17 & 150 \\ \hline
10 & 19 & 160 \\ \hline
10 & 23 & 180 \\ \hline
\end{array}
\end{displaymath}
\end{minipage}
\begin{minipage}[t]{.24\textwidth}
\begin{displaymath}
\begin{array}{|c|c|c|}
\hline
m & p & n_{1} \\ \hline
\hline
10 & 29 & 450 \\ \hline
10 & 31 & 340 \\ \hline
10 & 37 & 960 \\ \hline
10 & 41 & 250 \\ \hline
10 & 43 & 710 \\ \hline
10 & 47 & 370 \\ \hline
10 & 53 & 830 \\ \hline
10 & 59 & 360 \\ \hline
10 & 61 & 460 \\ \hline
10 & 67 & 480 \\ \hline
10 & 71 & 930 \\ \hline
10 & 73 & 330 \\ \hline
10 & 79 & 1270 \\ \hline
10 & 83 & 350 \\ \hline
10 & 89 & 260 \\ \hline
10 & 97 & 1930 \\ \hline
10 & 101 & 520 \\ \hline
10 & 103 & 280 \\ \hline
10 & 107 & 2130 \\ \hline
10 & 109 & 590 \\ \hline
\hline
11 & 13 & 88 \\ \hline
11 & 17 & 165 \\ \hline
11 & 19 & 176 \\ \hline
11 & 23 & 198 \\ \hline
11 & 29 & 495 \\ \hline
11 & 31 & 374 \\ \hline
11 & 37 & 1056 \\ \hline
11 & 41 & 275 \\ \hline
11 & 43 & 781 \\ \hline
11 & 47 & 407 \\ \hline
11 & 53 & 913 \\ \hline
11 & 59 & 396 \\ \hline
11 & 61 & 506 \\ \hline
11 & 67 & 528 \\ \hline
11 & 71 & 1023 \\ \hline
11 & 73 & 363 \\ \hline
11 & 79 & 1397 \\ \hline
11 & 83 & 385 \\ \hline
11 & 89 & 286 \\ \hline
11 & 97 & 2123 \\ \hline
\end{array}
\end{displaymath}
\end{minipage}
\begin{minipage}[t]{.24\textwidth}
\begin{displaymath}
\begin{array}{|c|c|c|}
\hline
m & p & n_{1} \\ \hline
\hline
11 & 101 & 572 \\ \hline
11 & 103 & 308 \\ \hline
11 & 107 & 2343 \\ \hline
11 & 109 & 649 \\ \hline
11 & 113 & 3795 \\ \hline
11 & 127 & 2486 \\ \hline
11 & 131 & 473 \\ \hline
\hline
12 & 17 & 180 \\ \hline
12 & 19 & 192 \\ \hline
12 & 23 & 216 \\ \hline
12 & 29 & 540 \\ \hline
12 & 31 & 408 \\ \hline
12 & 37 & 1152 \\ \hline
12 & 41 & 300 \\ \hline
12 & 43 & 852 \\ \hline
12 & 47 & 444 \\ \hline
12 & 53 & 996 \\ \hline
12 & 59 & 432 \\ \hline
12 & 61 & 552 \\ \hline
12 & 67 & 576 \\ \hline
12 & 71 & 1116 \\ \hline
12 & 73 & 396 \\ \hline
12 & 79 & 1524 \\ \hline
12 & 83 & 420 \\ \hline
12 & 89 & 312 \\ \hline
12 & 97 & 2316 \\ \hline
12 & 101 & 624 \\ \hline
12 & 103 & 336 \\ \hline
12 & 107 & 2556 \\ \hline
12 & 109 & 708 \\ \hline
12 & 113 & 4140 \\ \hline
12 & 127 & 2712 \\ \hline
12 & 131 & 516 \\ \hline
12 & 137 & 384 \\ \hline
12 & 139 & 1464 \\ \hline
12 & 149 & 1644 \\ \hline
12 & 151 & 6000 \\ \hline
12 & 157 & 3300 \\ \hline
\end{array}
\end{displaymath}
\end{minipage}
\end{center}
}
\normalsize{
\medskip
Link to the full list: {\tt https://github.com/blazejz/partitions} in section b.txt.

\newpage

\section*{Appendix B}

\addcontentsline{toc}{section}{Appendix B}

The smallest solution $n_{1}$ of the equation $c_{m}(n)\equiv 0\pmod p$ for $3\leq m\leq 12$ and prime numbers $m+2\leq p\leq m^{2}+m+1$.
}
\scriptsize{
\begin{center}
\begin{minipage}[t]{.24\textwidth}
\begin{displaymath}
\begin{array}{|c|c|c|}
\hline
m & p & n_{1} \\ \hline
\hline
3 & 5 & 34 \\ \hline
3 & 7 & 46 \\ \hline
3 & 11 & 82 \\ \hline
3 & 13 & 70 \\ \hline
\hline
4 & 7 & 61 \\ \hline
4 & 11 & 109 \\ \hline
4 & 13 & 93 \\ \hline
4 & 17 & 69 \\ \hline
4 & 19 & 197 \\ \hline
\hline
5 & 7 & 46 \\ \hline
5 & 11 & 61 \\ \hline
5 & 13 & 86 \\ \hline
5 & 17 & 101 \\ \hline
5 & 19 & 56 \\ \hline
5 & 23 & 171 \\ \hline
5 & 29 & 216 \\ \hline
5 & 31 & 76 \\ \hline
\hline
6 & 11 & 73 \\ \hline
6 & 13 & 103 \\ \hline
6 & 17 & 121 \\ \hline
6 & 19 & 67 \\ \hline
6 & 23 & 205 \\ \hline
6 & 29 & 259 \\ \hline
6 & 31 & 91 \\ \hline
6 & 37 & 343 \\ \hline
6 & 41 & 157 \\ \hline
6 & 43 & 109 \\ \hline
\hline
7 & 11 & 85 \\ \hline
7 & 13 & 120 \\ \hline
7 & 17 & 141 \\ \hline
7 & 19 & 78 \\ \hline
7 & 23 & 239 \\ \hline
7 & 29 & 302 \\ \hline
7 & 31 & 106 \\ \hline
7 & 37 & 400 \\ \hline
7 & 41 & 183 \\ \hline
7 & 43 & 127 \\ \hline
7 & 47 & 134 \\ \hline
7 & 53 & 211 \\ \hline
\hline
8 & 11 & 97 \\ \hline
\end{array}
\end{displaymath}
\end{minipage}
\begin{minipage}[t]{.24\textwidth}
\begin{displaymath}
\begin{array}{|c|c|c|}
\hline
m & p & n_{1} \\ \hline
\hline
8 & 13 & 137 \\ \hline
8 & 17 & 161 \\ \hline
8 & 19 & 89 \\ \hline
8 & 23 & 273 \\ \hline
8 & 29 & 345 \\ \hline
8 & 31 & 121 \\ \hline
8 & 37 & 457 \\ \hline
8 & 41 & 209 \\ \hline
8 & 43 & 145 \\ \hline
8 & 47 & 153 \\ \hline
8 & 53 & 241 \\ \hline
8 & 59 & 3201 \\ \hline
8 & 61 & 177 \\ \hline
8 & 67 & 617 \\ \hline
8 & 71 & 193 \\ \hline
8 & 73 & 281 \\ \hline
\hline
9 & 11 & 109 \\ \hline
9 & 13 & 154 \\ \hline
9 & 17 & 181 \\ \hline
9 & 19 & 100 \\ \hline
9 & 23 & 307 \\ \hline
9 & 29 & 388 \\ \hline
9 & 31 & 136 \\ \hline
9 & 37 & 514 \\ \hline
9 & 41 & 235 \\ \hline
9 & 43 & 163 \\ \hline
9 & 47 & 172 \\ \hline
9 & 53 & 271 \\ \hline
9 & 59 & 3601 \\ \hline
9 & 61 & 199 \\ \hline
9 & 67 & 694 \\ \hline
9 & 71 & 217 \\ \hline
9 & 73 & 316 \\ \hline
9 & 79 & 622 \\ \hline
9 & 83 & 334 \\ \hline
9 & 89 & 3331 \\ \hline
\hline
10 & 13 & 171 \\ \hline
10 & 17 & 201 \\ \hline
10 & 19 & 111 \\ \hline
10 & 23 & 341 \\ \hline
\end{array}
\end{displaymath}
\end{minipage}
\begin{minipage}[t]{.24\textwidth}
\begin{displaymath}
\begin{array}{|c|c|c|}
\hline
m & p & n_{1} \\ \hline
\hline
10 & 29 & 431 \\ \hline
10 & 31 & 151 \\ \hline
10 & 37 & 571 \\ \hline
10 & 41 & 261 \\ \hline
10 & 43 & 181 \\ \hline
10 & 47 & 191 \\ \hline
10 & 53 & 301 \\ \hline
10 & 59 & 4001 \\ \hline
10 & 61 & 221 \\ \hline
10 & 67 & 771 \\ \hline
10 & 71 & 241 \\ \hline
10 & 73 & 351 \\ \hline
10 & 79 & 691 \\ \hline
10 & 83 & 371 \\ \hline
10 & 89 & 3701 \\ \hline
10 & 97 & 4651 \\ \hline
10 & 101 & 951 \\ \hline
10 & 103 & 1611 \\ \hline
10 & 107 & 871 \\ \hline
10 & 109 & 621 \\ \hline
\hline
11 & 13 & 188 \\ \hline
11 & 17 & 221 \\ \hline
11 & 19 & 122 \\ \hline
11 & 23 & 375 \\ \hline
11 & 29 & 474 \\ \hline
11 & 31 & 166 \\ \hline
11 & 37 & 628 \\ \hline
11 & 41 & 287 \\ \hline
11 & 43 & 199 \\ \hline
11 & 47 & 210 \\ \hline
11 & 53 & 331 \\ \hline
11 & 59 & 4401 \\ \hline
11 & 61 & 243 \\ \hline
11 & 67 & 848 \\ \hline
11 & 71 & 265 \\ \hline
11 & 73 & 386 \\ \hline
11 & 79 & 760 \\ \hline
11 & 83 & 408 \\ \hline
11 & 89 & 4071 \\ \hline
11 & 97 & 5116 \\ \hline
\end{array}
\end{displaymath}
\end{minipage}
\begin{minipage}[t]{.24\textwidth}
\begin{displaymath}
\begin{array}{|c|c|c|}
\hline
m & p & n_{1} \\ \hline
\hline
11 & 101 & 1046 \\ \hline
11 & 103 & 1772 \\ \hline
11 & 107 & 958 \\ \hline
11 & 109 & 683 \\ \hline
11 & 113 & 2245 \\ \hline
11 & 127 & 2982 \\ \hline
11 & 131 & 4995 \\ \hline
\hline
12 & 17 & 241 \\ \hline
12 & 19 & 133 \\ \hline
12 & 23 & 409 \\ \hline
12 & 29 & 517 \\ \hline
12 & 31 & 181 \\ \hline
12 & 37 & 685 \\ \hline
12 & 41 & 313 \\ \hline
12 & 43 & 217 \\ \hline
12 & 47 & 229 \\ \hline
12 & 53 & 361 \\ \hline
12 & 59 & 4801 \\ \hline
12 & 61 & 265 \\ \hline
12 & 67 & 925 \\ \hline
12 & 71 & 289 \\ \hline
12 & 73 & 421 \\ \hline
12 & 79 & 829 \\ \hline
12 & 83 & 445 \\ \hline
12 & 89 & 4441 \\ \hline
12 & 97 & 5581 \\ \hline
12 & 101 & 1141 \\ \hline
12 & 103 & 1933 \\ \hline
12 & 107 & 1045 \\ \hline
12 & 109 & 745 \\ \hline
12 & 113 & 2449 \\ \hline
12 & 127 & 3253 \\ \hline
12 & 131 & 5449 \\ \hline
12 & 137 & 1249 \\ \hline
12 & 139 & 5269 \\ \hline
12 & 149 & 577 \\ \hline
12 & 151 & 4249 \\ \hline
12 & 157 & 2785 \\ \hline
\end{array}
\end{displaymath}
\end{minipage}
\end{center}
}
\normalsize {}
\medskip
Link to the full list: {\tt https://github.com/blazejz/partitions} in section c.txt.

\newpage

\section*{Appendix C}

\addcontentsline{toc}{section}{Appendix C}

In the first and the third columns there are presented the values of $\rk_{g}(n,2,h)$ and $\rk_{g}(n,3,h)$, respectively for $n$ big enough and particular values of $h$.

In the second column there are values of $\rk(n,2,h)$ in the case of $b_{K}=b_{2}$ for $n\geq 2$ and certain values of $h$ of the form $h=8h'+4$.

%\scriptsize{
\small{
%\begin{center}
\begin{minipage}[t]{.35\textwidth}
\begin{displaymath}
\begin{array}{|c|c||c|c|}
\hline
h & \rk_{g}(n,2,h) & h & \rk_{g}(n,2,h) \\ \hline
\hline
3 & 2 & 23 & 8 \\ \hline
4 & >13 & 24 & >13 \\ \hline
5 & 4 & 25 & 9 \\ \hline
6 & 5 & 26 & 9 \\ \hline
7 & 4 & 27 & 9 \\ \hline
8 & >13 & 28 & >13 \\ \hline
9 & 6 & 29 & 9 \\ \hline
10 & 7 & 30 & 10 \\ \hline
11 & 6 & 31 & 9 \\ \hline
12 & >13 & 32 & >13 \\ \hline
13 & 7 & 33 & 9 \\ \hline
14 & 8 & 34 & 10 \\ \hline
15 & 7 & 35 & 9 \\ \hline
16 & >13 & 36 & >13 \\ \hline
17 & 8 & 37 & 9 \\ \hline
18 & 8 & 38 & 10 \\ \hline
19 & 8 & 39 & 10 \\ \hline
20 & >13 & 40 & >13 \\ \hline
21 & 8 & 41 & 10 \\ \hline
22 & 9 & & \\ \hline
\end{array}
\end{displaymath}
\end{minipage}
\begin{minipage}[t]{.2\textwidth}
\begin{displaymath}
\begin{array}{|c|c|}
\hline
h & \rk(n,2,h) \\ \hline
\hline
4 & \leq 2 \\ \hline
12 & \leq 5 \\ \hline
20 & \leq 7 \\ \hline
28 & \leq 8 \\ \hline
36 & \leq 8 \\ \hline
44 & \leq 9 \\ \hline
52 & \leq 9 \\ \hline
60 & \leq 10 \\ \hline
68 & \leq 10 \\ \hline
76 & \leq 10 \\ \hline
\end{array}
\end{displaymath}
\end{minipage}
\begin{minipage}[t]{.35\textwidth}
\begin{displaymath}
\begin{array}{|c|c||c|c|}
\hline
h & \rk_{g}(n,3,h)& h & \rk_{g}(n,3,h) \\ \hline
\hline
3 & 2 & 27 & 9 \\ \hline
4 & 4 & 28 & 8 \\ \hline
5 & 4 & 29 & 6 \\ \hline
6 & 4 & 30 & 9 \\ \hline
7 & 4 & 31 & 6 \\ \hline
8 & 5 & 32 & 8 \\ \hline
9 & 5 & 33 & 8 \\ \hline
10 & 5 & 34 & 7 \\ \hline
11 & 5 & 35 & 8 \\ \hline
12 & 6 & 36 & >9 \\ \hline
13 & 5 & 37 & 6 \\ \hline
14 & 6 & 38 & 7 \\ \hline
15 & 6 & 39 & 8 \\ \hline
16 & 6 & 40 & 9 \\ \hline
17 & 5 & 41 & 6 \\ \hline
18 & 7 & 42 & >9 \\ \hline
19 & 5 & 43 & 6 \\ \hline
20 & 7 & 44 & 8 \\ \hline
21 & 7 & 45 & >9 \\ \hline
22 & 7 & 46 & 8 \\ \hline
23 & 6 & 47 & 6  \\ \hline
24 & 9 & 48 & >9 \\ \hline
25 & 7 & 49 & 7 \\ \hline
26 & 7 & 50 & 8 \\ \hline
\end{array}
\end{displaymath}
\end{minipage}
%\end{center}
}
\normalsize{}
\begin{flushleft}
Link to the full list of $\rk_{g}(n,2,l)$: {\tt https://github.com/blazejz/partitions} in section $2$--ary.txt.\\
Link to the full list of $\rk(n,2,l)$: {\tt https://github.com/blazejz/partitions} in section $2$--ary--$2$--step.txt.\\
Link to the full list of $\rk_{g}(n,3,l)$: {\tt https://github.com/blazejz/partitions} in section $3$--ary.txt.
\end{flushleft}

\normalsize{

}

\end{document}